%% file: SCP_2011_July25.tex
\documentclass[final,leqno]{siamltex}

\usepackage{amsmath}
\usepackage{amssymb}
\usepackage{graphicx}
\usepackage{color}
\usepackage{ifpdf}

\usepackage{tikz}
\usetikzlibrary{calc}
\input{hpv_pgf_lib.tex}

\def\norm#1{\left\|#1\right\|}

\newcommand{\R}{\mathbb{R}}
\newcommand{\abs}[1]{\left\vert#1\right\vert}
\newcommand{\set}[1]{\left\{#1\right\}}
\newtheorem{remark}{Remark}
\newtheorem{algorithm}{\noindent\normalfont\textsc{Algorithm}}
\newtheorem{assumption}{\textbf{A}\!\!}

\newtheorem{assumptionm}{\textbf{A}\!\!}

\newcommand{\aref}[1]{\textbf{\textsc{A\ref{#1}}}}
\newenvironment{example}[1][{\normalfont\textit{Example.\theproposition.}}]
{\begin{trivlist}\item[\hskip \labelsep {\bfseries #1}]}{\hfill{\Large$\diamond$}\end{trivlist}}

\title{ADJOINT-BASED PREDICTOR-CORRECTOR SEQUENTIAL CONVEX PROGRAMMING FOR PARAMETRIC NONLINEAR OPTIMIZATION}

\author{Tran D. Quoc$^{*}$, Carlo Savorgnan$^{*}$ \and 
Moritz Diehl\thanks{Department of Electrical Engineering (ESAT-SCD) and
Optimization in Engineering Center (OPTEC), K.U. Leuven, Kasteelpark Arenberg
10, B-3001 Leuven, Belgium ({\tt \{quoc.trandinh, carlo.savorgnan, 
moritz.diehl\}@esat.kuleuven.be})}}

\begin{document}
\maketitle

\begin{abstract}
This paper proposes an algorithmic framework for solving parametric optimization problems which we call adjoint-based predictor-corrector sequential convex
programming. After presenting the algorithm, we prove a contraction estimate that guarantees the tracking performance of the algorithm. 
Two variants of this algorithm are investigated. The first one can be used to solve nonlinear programming problems while the second variant is aimed to treat
online parametric nonlinear programming problems. The local convergence of these variants is proved. 
An application to a large-scale benchmark problem that originates from nonlinear model predictive control of a hydro power plant is implemented to examine the
performance of the algorithms.
\end{abstract}



\pagestyle{myheadings}
\thispagestyle{plain}
\markboth{D.Q. TRAN, C. SAVORGNAN AND M. DIEHL}{ADJOINT-BASED PREDICTOR-CORRECTOR SCP}

\section{Introduction}\label{sec:intro}
In this paper, we consider a parametric nonconvex optimization problem of the form:
\begin{equation}\label{eq:param_prob}
\makeatletter
\def\tagform@#1{\maketag@@@{#1\@@italiccorr}}
\makeatother
\left\{\begin{array}{cl}
\displaystyle\min_{x\in\R^n} & f(x) \\
\textrm{s.t.}  & g(x) + M\xi  = 0,\\
& x\in\Omega,
\end{array}\right.\tag{$\mathrm{P}(\xi)$}
\end{equation}
where $f : \R^n\to \R$ is convex, $g:\R^n\to\R^m$ is nonlinear, $\Omega\subseteq\R^n$ is a nonempty, closed convex set, and the parameter
$\xi$ belongs to a given subset $\mathcal{P}\subseteq\R^p$.
Matrix $M\in\R^{m\times p}$ plays the role of embedding the parameter $\xi$ into the equality constraints in a linear way.
Throughout this paper, $f$ and $g$ are assumed to be differentiable on their domain. 
Problem \ref{eq:param_prob} includes many (parametric) nonlinear programming problems such as standard nonlinear programs, nonlinear second
order cone programs, nonlinear semidefinite programs \cite{Kanzow2005,Nocedal2006,Stingl2007}.
The theory of parametric optimization has been extensively studied in many research papers and monographs, see,
e.g. \cite{Bonnans1994,Gauvin1988,Robinson1980}.

This paper deals with the efficient calculation of approximate solutions to a sequence of problems of the form \ref{eq:param_prob}, where
the parameter $\xi$ is slowly varying.
In other words, for a sequence $\{\xi_k\}_{k\geq 0}$ such that $\norm{\xi_{k+1}-\xi_{k}}$ is small, we want to solve the
problems $\mathrm{P}(\xi_k)$ in an efficient way without requiring more accuracy than needed in the result.

In practice, sequences  of problems of the form \ref{eq:param_prob} arise in the framework of real-time optimization, moving horizon
estimation, online data assimilation as well as in nonlinear model predictive control (NMPC).
A practical obstacle in these applications is the time limitation imposed on solving the underlying optimization problem for each value of the
parameter.
Instead of solving completely a nonlinear program at each sample time \cite{Biegler2000,Biegler1991, Bock2000,Helbig1998}, 
several online algorithms approximately solve the underlining nonlinear optimization problem by performing the first iteration of exact
Newton, sequential quadratic programming (SQP), Gauss-Newton or interior point methods \cite{Diehl2002b,Ohtsuka2004,Zavala2010}. 
In \cite{Diehl2002b,Ohtsuka2004} the authors only consider the algorithms in the framework of SQP method. This approach has been proved to be efficient in
practice and is widely used in many applications \cite{Diehl2002}. 
Recently, Zavala and Anitescu \cite{Zavala2010} proposed an inexact Newton-type method for solving online optimization problems based on
the framework of generalized equations \cite{Bonnans1994,Robinson1980}.

Other related work considers practical problems which possess more general convexity structure such as second order cone and semidefinite
cone constraints, nonsmooth convexity \cite{Fares2002,Stingl2007}. In these applications, standard optimization methods may not perform
satisfactorily.
Many algorithms for nonlinear second order cone and nonlinear semidefinite programming have recently been proposed and found many
applications in robust optimal control, experimental design, and topology optimization, see, e.g.
\cite{Bauer2000,Fares2002,Freund2007,Kato2007,Stingl2007}. These approaches can be considered as generalization of the SQP method. 

\subsection{Contribution}\label{sec:motivation}
The contribution of this paper is twofold. 
We start our paper by proposing a generic framework for the \textit{adjoint-based predictor-corrector sequential convex programming}
(APCSCP) for parametric optimization and prove a main result of the stability of tracking error for this algorithm (Theorem \ref{th:contraction_estimate}). 
In the second part the theory is specialized to the non-parametric case where a single optimization problem is solved.
The local convergence of these variants is also proved.
Finally, we present a numerical application to large scale nonlinear model predictive control of a hydro power plant with $259$ state variables and
$10$ controls. The performance of our algorithms is  compared with a standard real-time Gauss-Newton method and a conventional model
predictive control (MPC) approach.

APCSCP is based on three main ideas: sequential convex programming, predictor-corrector path-following and adjoint-based optimization. We briefly
explain these methods in the following. 

\subsection{Sequential convex programming}\label{subsec:SCPS}
The sequential convex programming (SCP) method is a local nonconvex optimization technique. SCP solves a sequence of convex approximations of
the original problem by convexifying only the nonconvex parts and preserving the structures that can efficiently be exploited by convex
optimization techniques \cite{Boyd2004,Mattingley2010,Nesterov2004}. 
Note that this method is different from SQP methods where quadratic programs are used as approximations of the problem.
This approach is useful when the problem possesses general convex structures such as conic constraints, a cost function depending on matrix
variables or convex constraints resulting from a low level problem in multi-level settings \cite{Bauer2000,Diehl2006c,Stingl2007}. Due to the
complexity of these structures, standard optimization techniques such as SQP and Gauss-Newton-type methods may not be convenient to apply. 
In the context of nonlinear conic programming, SCP approaches have been proposed under the names sequential semidefinite programming (SSDP) or SQP-type methods
\cite{Correa2002,Fares2002,Freund2007,Kanzow2005,Kato2007,Stingl2007}.
It has been shown in \cite{Diehl2006} that the superlinear convergence is lost if the linear semidefinite programming subproblems in the
SSDP algorithm are convexified. In \cite{Lewis2008} the authors considered a nonlinear program in the framework of a composite minimization
problem, where the inner function is linearized to obtain a convex subproblem which is made strongly convex by adding a quadratic proximal term.

In this paper, following the work in \cite{Fares2002,Freund2003,Quoc2010,Quoc2009}, we apply the SCP approach to solve problem \ref{eq:param_prob}.
The nonconvex constraint $g(x) + M\xi = 0$ is linearized at each iteration to obtain a convex approximation. The resulting subproblems can
be solved by exploiting convex optimization techniques.

\subsection{Predictor-corrector path-following methods}\label{subsec:PC_method}
In order to illustrate the idea of the predictor-corrector path-following method \cite{Deuflhard2004,Zavala2010}, we consider the case
$\Omega \equiv \mathbb{R}^n$. The KKT system of problem \ref{eq:param_prob} can be written as $F(z; \xi) = 0$, where $ z =
(x,{y})$ is its primal-dual variable. 
The solution $z^{*}(\xi)$ that satisfies the KKT condition for a given $\xi$ is in general a smooth map. 
By applying the implicit function theorem, the derivative of $z^{*}(\cdot)$ is expressed as
\begin{equation*}
\frac{\partial{z}^{*}}{\partial{\xi}}(\xi) = - \left[\frac{\partial{F}}{\partial{z}}(z^{*}(\xi); \xi)\right]^{-1}\frac{\partial{F}}{\partial{\xi}}(z^{*}(\xi); \xi).
\end{equation*}
In the parametric optimization context, we might have solved a problem with parameter $\bar{\xi}$ with solution $\bar{z} = z^{*}(\bar{\xi})$ and want to solve the next problem for a new parameter
$\hat{\xi}$. The tangential predictor $\hat{z}$ for this new solution $z^{*}(\hat{\xi})$ is given by
\begin{equation*}
\hat{z} = z^{*}(\bar{\xi}) + \frac{\partial{z}^{*}}{\partial\xi}(\bar{\xi})(\hat{\xi} - \bar{\xi}) = z^{*}(\bar{\xi}) -
\left[\frac{\partial{F}}{\partial{z}}(z^{*}(\bar\xi);
\bar\xi)\right]^{-1}\frac{\partial{F}}{\partial{\bar\xi}}(z^{*}(\bar\xi); \bar\xi)(\hat{\xi} -\bar{\xi}).    
\end{equation*}
Note the similarity with one step of a Newton method. In fact, a combination of the tangential predictor and the corrector due to a Newton method proves to be useful in the case that $\bar{z}$ was not
the exact solution of $F(z; \bar\xi) = 0$, but only an approximation.
In this case, linearization at $(\bar{z}, \bar{\xi})$ yields a formula that one step of a \textit{predictor-corrector path-following method} needs to satisfy:
\begin{equation}\label{eq:predi-corr}
F(\bar{z}; \bar\xi) + \frac{\partial{F}}{\partial\xi}(\bar{z};\bar{\xi})(\hat{\xi}-\bar{\xi}) + \frac{\partial{F}}{\partial z}(\bar{z};\bar{\xi})(\hat{z}-\bar{z}) = 0. 
\end{equation}
Written explicitly, it delivers the solution guess $\hat{z}$ for the next parameter $\hat{\xi}$ as
\begin{equation*}
\hat{z} = \bar{z} \underbrace{
-\left[\frac{\partial{F}}{\partial{z}}(\bar{z};\bar\xi)\right]^{-1}\frac{\partial{F}}{\partial{\xi}}(\bar{z};\bar{\xi})
(\hat{\xi}-\bar{\xi})}_{=\Delta z_{\mathrm{predictor}}}
\underbrace{-\left[\frac{\partial{F}}{\partial{z}}(\bar{z};\bar\xi)\right]^{-1}F(\bar{z};\bar{\xi})}_{=\Delta z_{\mathrm{corrector}}} 
\end{equation*}
Note that when the parameter enters linearly into $F$, we can write 
\begin{equation*}
\frac{\partial{F}}{\partial{\xi}}(\bar{z}; \bar{\xi})(\hat{\xi}-\bar{\xi})  = F(\bar{z}; \hat{\xi}) - F(\bar{z};\bar{\xi}). 
\end{equation*}
Thus, equation~\eqref{eq:predi-corr} is reduced to
\begin{equation} \label{eq:smoothpredictor}
F(\bar{z}; \hat{\xi}) + \frac{\partial{F}}{\partial{z}}(\bar{z})(\hat{z}-\bar{z}) = 0.
\end{equation}
It follows that  the predictor-corrector step can be easily obtained by just applying one standard Newton step to the new problem $\mathrm{P}(\hat{\xi})$
initialized at the past solution guess $\bar{z}$, if we employed the parameter embedding in the problem formulation \cite{Diehl2002}.

Based on the above analysis, the predictor-corrector path-following method only performs the first iteration of the exact Newton method for each new problem.
In this paper, by applying the generalized equation framework \cite{Robinson1980,Rockafellar1997}, we generalize this idea to the case where
more general convex constraints are considered.
When the parameter does not enter linearly into the problem, we can always reformulate this problem as \ref{eq:param_prob} by using slack variables.
In this case, the derivatives with respect to these slack variables contain the information of the predictor term.
Finally, we notice that the real-time iteration scheme proposed in \cite{Diehl2002b} can be considered as a variant of the
above predictor-corrector method in the SQP context.

\subsection{Adjoint-based method}\label{subsec:adjoint}
From a practical point of view, most of the time spent on solving optimization problems resulting from simulation-based methods is needed to
evaluate the functions and their derivatives \cite{Bock1984}.
Adjoint-based methods rely on  the observation that it is not necessary to use exact Jacobian matrices of the constraints.
Moreover, in some applications, the time needed to evaluate all the derivatives of the functions exceeds the time available to compute the solution of the
optimization problem.
The adjoint-based Newton-type methods in \cite{Diehl2010,Griewank2002,Schlenkrich2010} can work with an inexact Jacobian matrix and only require an
exact evaluation of the Lagrange gradient using adjoint derivatives to form the approximate optimization subproblems in the algorithm. 
This technique still allows to converge to the exact solutions but can save valuable time in the online performance of the algorithm.
  
\subsection{A tutorial example}\label{subsec:example}
The idea of the APCSCP method is illustrated in the following simple example.
\begin{example}\label{ex:tutorial_example}(\textit{Tutorial example})
Let us consider a simple nonconvex parametric optimization problem:
\begin{equation}\label{eq:tutorial_example}
\min\left\{ -x_1 ~|~ x^2_1 + 2x_2 + 2 -4\xi = 0, ~x^2_1 - x^2_2 + 1 \leq 0, ~x \geq 0, ~x \in\mathbb{R}^2 \right\},
\end{equation}
where $\xi\in \mathcal{P} := \{\xi\in\mathbb{R} ~:~ \xi \geq 1.2 \}$ is a parameter.
After few calculations, we can show that $x^{*}_{\xi} = (2\sqrt{\xi -\sqrt{\xi}}, 2\sqrt{\xi} - 1)^T$ is a stationary point
of problem \eqref{eq:tutorial_example} which is also the uniquely global optimum. It is clear that problem \eqref{eq:tutorial_example} satisfies the
\textit{strong second order sufficient condition} (SSOSC) at $x^{*}_{\xi}$.

Note that the constraint $x^2_1 - x^2_2 + 1 \leq 0$ is convex and it can be written as a second order cone constraint $\norm{(x_1, 1)^T}_2 \leq x_2$. Let
us define $g(x) := x_1^2 + 2x_2 + 2$, $M := -4$ and $\Omega := \{x\in\mathbb{R}^2 ~|~ \norm{(x_1, 1)^T}_2 \leq x_2, ~ x\geq 0 \}$. Then, problem
\eqref{eq:tutorial_example} can be casted into the form of \ref{eq:param_prob}.

\begin{figure}[ht]
\centerline{\includegraphics[angle=0,height=3.6cm,width=13.0cm]{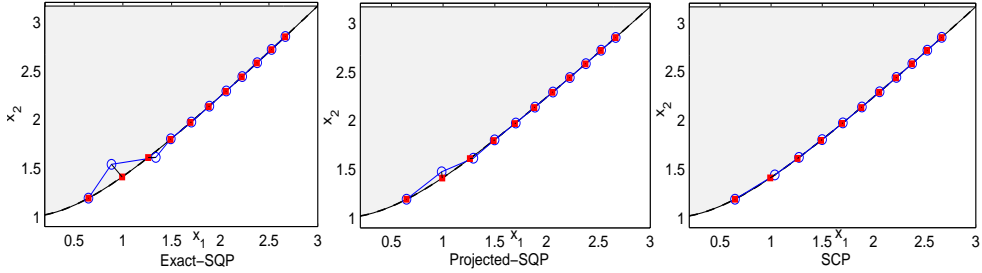}}
\caption{The trajectory of three methods $(k=0,\cdots, 9)$, $(${\color{red}$\diamond$} is $x^{*}(\xi_k)$ and {\color{blue}$\circ$} is $x^k$ $)$.}
\label{fig:exam_traject}
\end{figure}

\begin{figure}[ht]
\centerline{\includegraphics[angle=0,height=2.5cm,width=13.0cm]{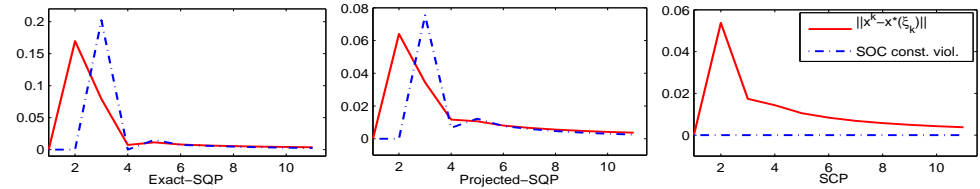}}
\caption{The tracking error and the cone constraint violation of three methods $(k=0,\cdots, 9)$.}
\label{fig:exam_res}
\end{figure}
 
The aim is to approximately solve problem \eqref{eq:tutorial_example} at each given value $\xi_k$ of the parameter $\xi$.
Instead of solving the nonlinear optimization problem at each $\xi_k$ until complete convergence, APCSCP only performs the first step of the
SCP  algorithm to obtain an approximate solution $x^k$ at $\xi_k$. Notice that the convex subproblem needed to be solved at each $\xi_k$ in the APCSCP method is
\begin{equation}\label{eq:sub_prob_exam}
\min_{x}\left\{ -x_1 ~|~ 2x_1^kx_1 + 2x_2 - (x_1^k)^2 + 2 - 4\xi = 0, ~\norm{(x_1, 1)^T} \leq x_2, ~ x\geq 0 \right\}.
\end{equation}
We compare this method with other known real-time iteration algorithms. The
first one is the real-time iteration with an exact SQP method and the second algorithm is the real-time iteration with an SQP method using a projected Hessian
\cite{Diehl2002b,Jarre2003}. 
In the second algorithm, the Hessian matrix of the Lagrange function is projected onto the cone of symmetric positive semidefinite matrices to obtain a convex
quadratic programming  subproblem.

Figures \ref{fig:exam_traject} and \ref{fig:exam_res} illustrate the performance of three methods when $\xi_k = 1.2 + k\Delta\xi_k$ for $k=0, \dots, 9$ and
$\Delta\xi_k = 0.25$. The initial point $x^0$ of three methods is chosen at the true solution of $\mathrm{P}(\xi_0)$.
We can see that the performance of the exact SQP and the SQP using projected Hessian is quite similar. However, the second order cone constraint $\norm{(x_1,
1)^T}_2\leq x_2$ is violated in both methods.  
The SCP method preserves the feasibility and better follows the exact solution trajectory. Note that the subproblem in the exact SQP method is a nonconvex
quadratic program, a convex QP in the projected SQP case and a second order cone constrained program \eqref{eq:sub_prob_exam} in the SCP method.
\end{example}
   
\subsection{Notation}\label{subsec:notation}
Throughout this paper, we use the notation $\nabla f$ for the gradient vector of a scalar function $f$, $g'$ for the Jacobian matrix of a
vector valued function $g$ and $\mathcal{S}^n$ (resp., $\mathcal{S}^n_{+}$
and $\mathcal{S}^n_{++}$) for the set of $n\times n$ real symmetric (resp., positive semidefinite and positive definite) matrices.
The notation $\norm{\cdot}$ stands for the Euclidean norm.
The ball $\mathcal{B}(x,r)$ of radius $r$ centered at $x$ is defined as $\mathcal{B}(x,r) := \{y\in\mathbb{R}^n~|~ \norm{y-x} < r\}$ and
$\bar{\mathcal{B}}(x,r)$ is its closure.

The rest of this paper is organized as follows. 
Section \ref{sec:adjSCPalg} presents a generic framework of the \textit{adjoint-based predictor-corrector SCP algorithm} (APCSCP).
Section \ref{sec:contraction_estimate} proves the local contraction estimate for APCSCP and the stability of the approximation error. 
Section \ref{sec:SCP_case} considers an \textit{adjoint-based SCP algorithm} for solving nonlinear programming problems as a special case. 
The last section presents computational results for an application of the proposed algorithms in nonlinear model predictive control (NMPC) of
a hydro power plant. 

\section{An adjoint-based predictor-corrector SCP algorithm}\label{sec:adjSCPalg}
In this section, we present a generic algorithmic framework for solving the parametric optimization problem \ref{eq:param_prob}.
Traditionally, at each sample $\xi_k$ of parameter $\xi$, a nonlinear program $\mathrm{P}(\xi_k)$ is solved to get a completely converged
solution $\bar{z}(\xi_k)$. Exploiting the real-time iteration idea \cite{Diehl2002,Diehl2002b}, in our algorithm below, only one convex
subproblem is solved to get an approximated solution $z^k$ at $\xi_k$ to $\bar{z}(\xi_k)$. 
   
Suppose that $z^k := (x^k, {y}^k) \in\Omega\times \R^m$ is a given KKT point of $\mathrm{P}(\xi_k)$ (more details can be found in the next
section), $A_k$ is a given $m\times n$ matrix and $H_k\in\mathcal{S}^n_{+}$. We consider the following parametric optimization subproblem: 
\begin{equation}\label{eq:convex_subprob}
\makeatletter
\def\tagform@#1{\maketag@@@{#1\@@italiccorr}}
\makeatother
\left\{\begin{array}{cl}
\displaystyle
\min_{x\in\R^n} &\left\{ f(x) + (m^k)^T(x-x^k) + \frac{1}{2}(x-x^k)^TH_k(x-x^k) \right\}\\
\textrm{s.t.} &A_k(x-x^k) + g(x^k) + M\xi = 0, \\
       &x \in \Omega,
\end{array}\right.\tag{$\mathrm{P}(z^k,A_k,H_k;\xi)$}
\end{equation}
where $m^k := m(z^k, A_k) = \left(g'(x^k)-A_k\right)^Ty^k$. Matrix $A_k$ is an approximation to  $g'(x^k)$ at $x^k$, $H_k$ is a regularization or an
approximation to $\nabla^2_x\mathcal{L}(\bar{z}^k)$, where $\mathcal{L}$ is the Lagrange function of \ref{eq:param_prob} to be defined in Section
\ref{sec:contraction_estimate}. Vector $m^k$ can be considered as a correction term of the inconsistency between $A_k$ and $g'(x^k)$. 
Vector $y^k$ is referred to as the Lagrange multiplier.
Since $f$ and $\Omega$ are convex and $H_k$ is symmetric positive semidefinite, the subproblem \ref{eq:convex_subprob} is convex. Here, $z^k$,
$A_k$ and $H_k$ are considered as parameters.

\begin{remark}\label{re:adjoint}
Note that computing the term $g'(x^k)^Ty^k$ of the correction vector $m^k$ does not require the whole Jacobian matrix $g'(x^k)$, which is
usually time consuming to evaluate. This adjoint directional derivative can be cheaply evaluated  by using  adjoint
methods \cite{Griewank2000}.
\end{remark}

The adjoint-based predictor-corrector SCP algorithmic framework is described as follows.

\noindent\rule[1pt]{\textwidth}{1.0pt}{~~}
\begin{algorithm}\label{alg:A1} 
$\mathrm{(}${\textit{Adjoint-based predictor-corrector SCP algorithm}}~$\mathrm{(APCSCP)})$.
\end{algorithm}
\vskip -0.15cm
\noindent\rule[1pt]{\textwidth}{0.5pt}
\vskip -0.15cm
\begin{romannum}
\item[\textbf{Initialization.}] For a given parameter $\xi_0\in\mathcal{P}$,  
solve approximately (off-line) $\mathrm{P}(\xi_0)$ to get an approximate KKT point $z^0:=(x^0,{y}^0)$.
Compute $g(x^0)$, find a matrix $A_0$ which approximates $g'(x^0)$ and $H_0\in\mathcal{S}^n_{+}$. Then, compute vector
$m^0:=\left(g'(x^0)-A_0\right)^T{y}^0$. 
Set $k:=0$.
\item[\textbf{Iteration $k$ ($k=0,1,\dots $})] For a given $(z^k, A_k, H_k)$, perform the three steps below:
\begin{remunerate}
\item[\textit{Step 1. }] Get a new parameter value $\xi_{k+1} \in\mathcal{P}$.
\item[\textit{Step 2. }] Solve the convex subproblem $\mathrm{P}(z^k,A_k, H_k;\xi_{k+1})$ to obtain a solution $x^{k+1}$ and the
corresponding multiplier $y^{k+1}$. 
\item[\textit{Step 3. }] Evaluate
$g(x^{k+1})$, update (or recompute) matrices $A_{k+1}$ and $H_{k+1}\in \mathcal{S}^n_{+}$. Compute vector $m^{k+1} :=
g'(x^{k+1})^Ty^{k+1} -
A_{k+1}^Ty^{k+1}$.
Set $k := k+1$ and go back to Step 1.
\end{remunerate}
\end{romannum}
\noindent\rule[1pt]{\textwidth}{1.0pt}

The core step of Algorithm \ref{alg:A1} is to solve the convex subproblem \ref{eq:convex_subprob} at each
iteration. To reduce the computational time, we can either implement an optimization method which exploits the structure of the problem or
rely on several efficient software tools that are available for convex optimization \cite{Boyd2004,Nesterov2004,Nocedal2006}.
In this paper, we are most interested in the case where one evaluation of $g'$ is very expensive.
A possibly simple choice of $H_k$ is $H_k = 0$ for all $k\geq 0$. 

The initial point $z^0$ is obtained by solving off-line $\mathrm{P}(\xi_0)$.
However, as we will show later [Corollary \ref{co:tracking_corollary}], if we choose $z^0$ close to the set of KKT points $Z^{*}(\xi_0)$ of
$\mathrm{P}(\xi_0)$ (not necessarily an exact solution) then the new KKT point $z^1$ of $\mathrm{P}(z^0, A_0, H_0;\xi^1)$ is still close to
$Z^{*}(\xi_1)$ of $\mathrm{P}(\xi_1)$ provided that $\norm{\xi_1-\xi_0}$ is sufficiently small. Hence, in practice, we only need to solve
approximately problem $\mathrm{P}(\xi_0)$ to get a starting point $z^0$. 

In the NMPC framework, the parameter $\xi$ usually coincides with the initial state of the dynamic system at the current time of the moving
horizon. If matrix $A_k\equiv g'(x^k)$, the exact Jacobian matrix of $g$ at $x^k$ and $H_k \equiv 0$, then this algorithm collapses to the 
\textit{real-time SCP
method} (RTSCP) considered in \cite{Quoc2009}. 

\section{Contraction estimate}\label{sec:contraction_estimate}
In this section, we will show that under certain assumptions, the sequence $\{z^k\}_{k\geq 0}$ generated by Algorithm \ref{alg:A1} remains close to
the sequence of the true KKT points $\{\bar{z}_k\}_{k\geq 0}$ of problem $\mathrm{P}(\xi_k)$.
Without loss of generality, we assume that the objective function $f$ is linear, i.e. $f(x) = c^Tx$, where $c\in\mathbb{R}^n$ is
given. Indeed, since $f$ is convex, by using a slack variable $s$, we can reformulate \ref{eq:param_prob} as a nonlinear
program $\min_{(x,s)}\big\{s ~|~ g(x) + M\xi = 0, ~x\in \Omega, ~f(x) \leq s\big\}$.
 
\subsection{KKT condition as a generalized equation}
Let us first define the Lagrange function of problem $\mathrm{P}(\xi)$ as
\begin{equation*}
\mathcal{L}(x,{y};\xi) := c^Tx + (g(x) + M\xi)^Ty,  
\end{equation*}
where $y$ is the Lagrange multiplier associated with the constraint $g(x) + M\xi=0$.
Since the constraint $x\in\Omega$ is convex and  implicitly represented, we will consider it separately.
The KKT condition for $\mathrm{P}(\xi)$ is now written as
\begin{equation}\label{eq:P_kkt}
\begin{cases}
0 \in c + g'(x)^Ty + \mathcal{N}_{\Omega}(x),\\
0 = g(x) + M\xi,
\end{cases}
\end{equation}
where $\mathcal{N}_{\Omega}(x)$ is the normal cone of $\Omega$ at $x$ defined as
\begin{equation}\label{eq:P_normal_cone}
\mathcal{N}_{\Omega}(x) := \begin{cases}
\left\{ u\in\R^n ~|~ u^T(x-v) \geq 0, ~v\in\Omega\right\}, ~~\text{if}~~ x\in\Omega\\
\emptyset, ~~\text{otherwise}.
\end{cases} 
\end{equation}
Note that the first line of \eqref{eq:P_kkt} implicitly includes the constraint $x\in\Omega$.

A pair $(\bar{x}(\xi),\bar{y}(\xi))$ satisfying \eqref{eq:P_kkt} is called a KKT point of $\mathrm{P}(\xi)$ and $\bar{x}(\xi)$ is called a
stationary point of $\mathrm{P}(\xi)$ with the corresponding multiplier $\bar{y}(\xi)$. Let us denote by $Z^{*}(\xi)$ and $X^{*}(\xi)$ the
set of KKT points and the set of stationary points of $\mathrm{P}(\xi)$, respectively.
In the sequel, we use the letter $z$ for the pair of $(x,y)$, i.e. $z := (x^T, y^T)^T$. 

Throughout this paper, we require the following assumptions which are standard in optimization.
\begin{assumption}\label{as:A1}
The function $g$ is twice differentiable on their domain.
\end{assumption}

\begin{assumption}\label{as:A2}
For a given $\xi_0\in\mathcal{P}$, problem $\mathrm{P}(\xi_0)$ has at least one KKT point $\bar{z}^0$, i.e. $Z^{*}(\xi_0)\neq\emptyset$.
\end{assumption}

Let us define  
\begin{equation}\label{eq:P_varphi}
F(z) := \begin{pmatrix}c + g'(x)^Ty\\ g(x)\end{pmatrix}, 
\end{equation}
and $K := \Omega\times\R^m$.
Then, the KKT condition \eqref{eq:P_kkt} can be expressed in terms of a parametric generalized equation as follows:
\begin{equation}\label{eq:P_gen_eq}
0 \in F(z) + C\xi + \mathcal{N}_K(z), 
\end{equation}
where $C := \big[\begin{smallmatrix}0\\ M\end{smallmatrix}\big]$.
Generalized equations are an essential tool to study many problems in nonlinear analysis, perturbation analysis, variational calculations as
well as optimization \cite{Bonnans2000,Klatte2002,Rockafellar1997}.

Suppose that, for some $\xi_k\in\mathcal{P}$, the set of KKT points $Z^{*}(\xi_k)$ of $\mathrm{P}(\xi_k)$ is nonempty. 
For any fixed $\bar{z}^k\in Z^{*}(\xi_k)$, we define the following set-valued mapping:
\begin{equation}\label{eq:P_linearization}
L(z;\bar{z}^k, \xi_k) := F(\bar{z}^k) + F'(\bar{z}^k)(z-\bar{z}^k) + C\xi_k + \mathcal{N}_K(z). 
\end{equation}
We also define the inverse mapping $L^{-1} : \R^{n+m}\to \R^{n+m}$ of $L(\cdot;\bar{z}^k,\xi_k)$ as follows:
\begin{equation}\label{eq:P_Tinverse}
L^{-1}(\delta;\bar{z}^k,\xi_k) := \left\{ z\in\R^{n+m} ~:~ \delta \in L(z;\bar{z}^k, \xi_k)\right\}. 
\end{equation}

Now, we consider the KKT condition of the subproblem \ref{eq:convex_subprob}.
For given neighborhoods $\mathcal{B}(\bar{z}^k, r_z)$ of $\bar{z}^k$ and $\mathcal{B}(\xi_k, r_{\xi})$ of $\xi_k$, and
 $z^k\in\mathcal{B}(\bar{z}^k, r_z)$, $\xi_{k+1}\in \mathcal{B}(\xi_k, r_{\xi})$ and given matrices $A_k$ and $H_k\in\mathcal{S}^n_{+}$, let
us consider the convex subproblem $\mathrm{P}(z^k,A_k, H_k;\xi_{k+1})$ with respect to the parameter $(z^k, A_k, H_k, \xi_{k+1})$. 
The KKT condition of this problem is expressed as follows. 
\begin{equation}\label{eq:kkt_subprob}
\begin{cases}
0 \in c + m(z^k,A_k) + H_k(x-x^k) + A_k^T{y} + \mathcal{N}_{\Omega}(x),\\
0 = g(x^k) + A_k(x-x^k) + M\xi_{k+1},
\end{cases} 
\end{equation}
where $\mathcal{N}_{\Omega}(x)$ is defined by \eqref{eq:P_normal_cone}. 
Suppose that the Slater constraint qualification holds for the subproblem $\mathrm{P}(z^k,A_k, H_k;\xi_{k+1})$, i.e.:
 \begin{equation*}
\text{ri}(\Omega)\cap\set{x\in\R^n ~|~  g(x^k) + A_k(x- x^k) + M\xi_{k+1} = 0}\neq \emptyset,
\end{equation*}
where $\text{ri}(\Omega)$ is the relative interior of $\Omega$.
Then by convexity of $\Omega$, a point $z^{k+1} := (x^{k+1}, y^{k+1})$ is a KKT point of $\mathrm{P}(z^k,A_k, H_k;\xi_{k+1})$ if and only if
$x^{k+1}$ is a solution to $\mathrm{P}(z^k,A_k, H_k;\xi_{k+1})$ associated with the multiplier ${y}^{k+1}$.

Since $g$ is twice differentiable by Assumption \aref{as:A1} and $f$ is linear, for a given $z = (x, y)$, we have 
\begin{equation}\label{eq:P_E11_term}
\nabla^2_x\mathcal{L}(z) = \sum_{i=1}^my_i\nabla^2 g_i(x), 
\end{equation}
the Hessian matrix of the Lagrange function $\mathcal{L}$, where $\nabla^2 g_i(\cdot)$ is the Hessian matrix of $g_i$ ($i=1,\dots, m$).
Let us define the following matrix:
\begin{equation}\label{eq:H_matrix}
\tilde{F}'_k := \begin{bmatrix}H_k & A_k^T \\ A_k & 0\end{bmatrix}, 
\end{equation}
where $H_k \in \mathcal{S}^n_{+}$.
The KKT condition \eqref{eq:kkt_subprob} can be written as a parametric linear generalized equation:
\begin{equation}\label{eq:gen_eq_for_subprob}
0 \in F(z^k) + \tilde{F}'_k(z-z^k) + C\xi_{k+1} + \mathcal{N}_K(z), 
\end{equation}
where $z^k$, $\tilde{F}'_k$ and $\xi_{k+1}$ are considered as parameters. Note that if $A_k = g'(x^k)$ and $H_k =
\nabla^2_x\mathcal{L}(z^k)$ then \eqref{eq:gen_eq_for_subprob} is the linearization of the nonlinear generalized equation
\eqref{eq:P_gen_eq} at $(z^k, \xi_{k+1})$ with respect to $z$.

\begin{remark}\label{re:predictor_corrector}
Note that \eqref{eq:gen_eq_for_subprob} is a generalization of \eqref{eq:smoothpredictor}, where the approximate Jacobian $\tilde{F}'_k$ is used instead of the exact one. Therefore, \eqref{eq:gen_eq_for_subprob} can be viewed as one iteration of the inexact predictor-corrector path-following method for solving \eqref{eq:P_gen_eq}.
\end{remark}
 
\subsection{The strong regularity concept}
We recall the following definition of the \textit{strong regularity} concept. 
This definition can be considered as the strong regularity of the generalized equation \eqref{eq:P_gen_eq} in the context of nonlinear
optimization, see \cite{Robinson1980}.
 
\begin{definition}\label{de:strong_regularity}
Let $\xi_k\in\mathcal{P}$ such that the set of KKT points $Z^{*}(\xi_k)$ of $\mathrm{P}(\xi_k)$ is nonempty. Let $\bar{z}^k\in Z^{*}(\xi_k)$
be a given KKT point of $\mathrm{P}(\xi_k)$. Problem $\mathrm{P}(\xi_k)$ is said to be \textit{strongly regular} at $\bar{z}^k$ if there
exist neighborhoods $\mathcal{B}(0, \bar{r}_{\delta})$ of the origin and $\mathcal{B}(\bar{z}^k, \bar{r}_z)$ of $\bar{z}^k$ such that the
mapping $z_k^{*}(\delta) := \mathcal{B}(\bar{z}^k, \bar{r}_z) \cap L^{-1}(\delta;\bar{z}^k, \xi_k)$ is single-valued and Lipschitz
continuous in $\mathcal{B}(0, \bar{r}_{\delta})$ with a Lipschitz constant $0 < \gamma < +\infty$, i.e.
\begin{equation}\label{eq:P_Lipschitz}
\norm{z^{*}_k(\delta) - z^{*}_k(\delta')} \leq \gamma\norm{\delta-\delta'}, ~~\forall \delta, \delta'\in\mathcal{B}(0, \bar{r}_{\delta}). 
\end{equation}
\end{definition}
Note that the constants $\gamma$, $\bar{r}_z$ and $\bar{r}_{\delta}$ in Definition \ref{de:strong_regularity} are global and do not
depend on the index $k$.

From the definition of $L^{-1}$ where strong regularity holds, there exists a unique $z^{*}_k(\delta)$ such that $\delta \in F(\bar{z}^k) + F'(\bar{z}^k)(z^{*}_k(\delta)-\bar{z}^k) + C\xi_k + \mathcal{N}_K(z^{*}_k(\delta))$. Therefore,
\begin{eqnarray*}
z^{*}_k(\delta) ~&&= (F'(\bar{z}^k) + \mathcal{N}_K)^{-1}\left(F'(\bar{z}^k)\bar{z}^k - F(\bar{z}^k) - C\xi_k + \delta\right)\\
&& = \bar{J}_k\left(F'(\bar{z}^k)\bar{z}^k - F(\bar{z}^k) - C\xi_k + \delta\right), 
\end{eqnarray*}
where $\bar{J}_k := (F'(\bar{z}^k) + \mathcal{N}_K)^{-1}$. The strong regularity of $\mathrm{P}(\xi)$ at $\bar{z}^k$ is equivalent to the
single-valuedness and the Lipschitz continuity of $\bar{J}_k$ around $v^k := F'(\bar{z}^k)\bar{z}^k - F(\bar{z}^k) - C\xi_k$.

The strong regularity concept is widely used in variational analysis, perturbation analysis as well as in
optimization \cite{Bonnans2000,Klatte2002,Rockafellar1997}. 
In view of optimization, strong regularity implies the strong second order sufficient optimality condition (SSOSC) if the linear
independence constraint qualification (LICQ) holds \cite{Robinson1980}.
If the convex set $\Omega$ is polyhedral and the LICQ holds, then strong regularity is equivalent to SSOSC \cite{Dontchev1996}.
In order to interpret the strong regularity condition of $\mathrm{P}(\xi^k)$ at $\bar{z}^k\in Z^{*}(\xi_k)$ in terms of perturbed optimization, we consider the
following optimization problem 
\begin{equation}\label{eq:pert_cvp}
\left\{\begin{array}{cl}
\displaystyle\min_{x\in\mathbb{R}^n} &(c - \delta_c)^Tx + \frac{1}{2}(x- \bar{x}^k)^T\nabla^2_x\mathcal{L}(\bar{x}^k,\bar{y}^k)(x-\bar{x}^k) \\
\mathrm{s.t.}~ &g(\bar{x}^k) + g'(\bar{x}^k)(x-\bar{x}^k) + M\xi_k = \delta_g,\\
& x \in \Omega.
\end{array}
\right. 
\end{equation}
Here, $\delta = (\delta_c, \delta_g) \in \mathcal{B}(0, \bar{r}_{\delta})$ is a perturbation.
Problem $\mathrm{P}(\xi_k)$ is \textit{strongly regular} at $\bar{z}^k$ if and only if \eqref{eq:pert_cvp} has a unique KKT
point $z^{*}_k(\delta)$ in $\mathcal{B}(\bar{z}^k,\bar{r}_z)$ and $z^{*}_k(\cdot)$ is Lipschitz continuous in
$\mathcal{B}(0,\bar{r}_{\delta})$ with a Lipschitz constant $\gamma$.   

\vskip0.1cm
\begin{example}\label{ex:example2} 
Let us recall example \eqref{eq:tutorial_example} in Section \ref{sec:motivation}. The optimal multipliers associated with two constraints $x^2_1 + 2x_2 + 2
- 4\xi = 0$ and $x^2_1 - x_2^2 +1 \leq 0$ are $y^{*}_1 = (2\sqrt{\xi}-1)[8\sqrt{\xi^2 - \xi\sqrt{\xi}}]^{-1} > 0$ and $y^{*}_2
= [8\sqrt{\xi^2 - \xi\sqrt{\xi}}]^{-1} > 0$, respectively. Since the last inequality constraint is active while $x\geq 0$ is inactive, we can easily
compute the critical cone as  $\mathcal{C}(x^{*}_{\xi},y^{*}) = \{(d_1,0)\in\mathbb{R}^2 ~|~ x^{*}_{\xi 1}d_1 = 0\}$. The Hessian matrix
$\nabla_x^2\mathcal{L}(x^{*}_{\xi},y^{*}) =
\left[\begin{smallmatrix}2(y^{*}_1+y^{*}_2) & 0\\ 0 & -2y^{*}_2\end{smallmatrix}\right]$ of the Lagrange function $\mathcal{L}$ is positive definite in
$\mathcal{C}(x^{*}_{\xi},y^{*})$. Hence, the second order sufficient optimality condition for \eqref{ex:tutorial_example} is satisfied. Moreover, $y_2^{*} > 0$
which says that the strict complementarity condition holds. Therefore, problem \eqref{ex:tutorial_example} satisfies the the strong second order sufficient
condition.  
On the other hand, it is easy to check that the LICQ condition holds for \eqref{ex:tutorial_example} at $x^{*}_{\xi}$.
By applying \cite[Theorem 4.1]{Robinson1980}, we can conclude that \eqref{eq:tutorial_example} is strongly regular at $(x^{*}_{\xi}, y^{*})$.
\end{example}

The following lemma shows the nonemptiness of $Z^{*}(\xi)$ in the neighborhood of $\xi_k$. 

\begin{lemma}\label{le:nonemptyness_of_Gamma}
Suppose that Assumption \aref{as:A1} is satisfied and $Z^{*}(\xi_k)$ is nonempty for a given $\xi_k\in\mathcal{P}$. 
Suppose further that problem $\mathrm{P}(\xi_k)$ is strongly regular at $\bar{z}^k$ for a given $\bar{z}^k \in Z^{*}(\xi_k)$.
Then there exist neighborhoods $\mathcal{B}(\xi_k, r_{\xi})$ of $\xi_k$ and $\mathcal{B}(\bar{z}^k, r_z)$ of $\bar{z}^k$ such that
$Z^{*}(\xi_{k+1})$ is nonempty for all $\xi_{k+1}\in\mathcal{B}(\xi_k, r_{\xi})$ and $Z^{*}(\xi_{k+1})\cap\mathcal{B}(\bar{z}^k, r_z)$
contains only one point $\bar{z}^{k+1}$. Moreover, there exists a constant $0 \leq \bar{\sigma} < +\infty$ such that:
\begin{equation}\label{eq:lm31_estimate}
\norm{\bar{z}^{k+1}-\bar{z}^k} \leq \bar{\sigma}\norm{\xi_{k+1}-\xi_k}. 
\end{equation}
\end{lemma}

\begin{proof}
Since the KKT condition of $\mathrm{P}(\xi_k)$ is equivalent to the generalized equation \eqref{eq:P_gen_eq} with $\xi=\xi_k$. By applying
\cite[Theorem 2.1]{Robinson1980} we conclude that there exist neighborhoods $\mathcal{B}(\xi_k, r_{\xi})$ of $\xi_k$
and $\mathcal{B}(\bar{z}^k, r_z)$ of $\bar{z}^k$ such that $Z^{*}(\xi_{k+1})$ is nonempty for all $\xi_{k+1}\in\mathcal{B}(\xi_k, r_{\xi})$
and $Z^{*}(\xi_{k+1})\cap\mathcal{B}(\bar{z}^k, r_z)$ contains only one point $\bar{z}^{k+1}$.
On the other hand, since $\norm{F(\bar{z}^k) + C\xi_k - F(\bar{z}^k) - C\xi_{k+1}} = \norm{M(\xi_k-\xi_{k+1})} \leq
\norm{M}\norm{\xi_{k+1} - \xi_k}$, by using the formula \cite[2.4]{Robinson1980}, we obtain the estimate \eqref{eq:lm31_estimate}.
\end{proof}

\subsection{A contraction estimate for APCSCP using an inexact Jacobian matrix}
In order to prove a contraction estimate for APCSCP, throughout this section, we make the following assumptions.

\begin{assumption}\label{as:A3} For a given $\bar{z}^k \in Z^{*}(\xi_k)$, $k\geq 0$, the following conditions are satisfied.
\begin{itemize}
\item[$\mathrm{a)}$] There exists a constant $0 \leq \kappa < \frac{1}{2\gamma}$ such that:
\begin{equation}\label{eq:A3_condition}
\norm{F'(\bar{z}^k) - \tilde{F}'_k} \leq \kappa,
\end{equation}
where $\tilde{F}_k'$ is defined by \eqref{eq:H_matrix}.

\item[$\mathrm{b)}$] The Jacobian mapping $F'(\cdot)$ is Lipschitz continuous on $\mathcal{B}(\bar{z}^k, r_z)$ around $\bar{z}^k$, i.e.
there exists a
constant $0 \leq  \omega < +\infty$ such that:
\begin{equation}\label{eq:A4_condition}
\norm{F'(z) - F'(\bar{z}^k)} \leq \omega\norm{z-\bar{z}^k}, ~\forall  z\in \mathcal{B}(\bar{z}^k, r_z).
\end{equation}
\end{itemize}
\end{assumption}

Note that Assumption \aref{as:A3} is commonly used in the theory of Newton-type and Gauss-Newton methods \cite{Deuflhard2004,Diehl2005},
where the residual term is required to be sufficiently small in a neighborhood of the local solution. 
From the definition of $\tilde{F}_k'$ we have
\begin{equation*}
F'(\bar{z}^k) - \tilde{F}_k' = \begin{bmatrix} \nabla_x^2\mathcal{L}(\bar{z}^k) - H_k &
g'(\bar{x}^k)^T - A_k^T\\
g'(\bar{x}^k) - A_k & O \end{bmatrix}. 
\end{equation*}
Hence, $\norm{F'(\bar{z}^k) - \tilde{F}_k'}$ depends on the norms of  $\nabla_x^2\mathcal{L}(\bar{z}^k) - H_k$ and $g'(\bar{x}^k) - A_k$.
These quantities are the error of the approximations $H_k$ and $A_k$ to the Hessian matrix $\nabla_x^2\mathcal{L}(\bar{z}^k)$ and the
Jacobian matrix $g'(\bar{x}^k)$, respectively.
On the one hand, Assumption \aref{as:A3}a) requires the positive definiteness of $H_k$ to be an approximation of $\nabla^2_x\mathcal{L}$ (which is not
necessarily positive definite).
On the other hand, it requires that matrix $A_k$ is a sufficiently good approximation to the Jacobian matrix $g'$ in the neighborhood of the
stationary point $\bar{x}^k$. Note that the matrix $H_k$ in the Newton-type method proposed in \cite{Bonnans1994} is not necessarily positive definite.

Now, let us define the following mapping:
\begin{equation}\label{eq:inverse_operator}
J_k := (\tilde{F}'_k + \mathcal{N}_K)^{-1}, 
\end{equation}
where $\tilde{F}'_k$ is defined by \eqref{eq:H_matrix}.
The lemma below shows that $J_k$ is single-valued and Lipschitz continuous in a neighbourhood of $\bar{v}^k := \tilde{F}'_k\bar{z}^k - F(\bar{z}^k) - C\xi_k$.

\begin{lemma}\label{le:inverse_operator}
Suppose that Assumptions \aref{as:A1}, \aref{as:A2} and \aref{as:A3}$\mathrm{a)}$ are satisfied. 
Then there exist neighborhoods $\mathcal{B}(\xi_k, r_{\xi})$ and $\mathcal{B}(\bar{z}^k, r_z)$ such that if we take
any $z^k\in\mathcal{B}(\bar{z}^k, r_z)$ and $\xi_{k+1}\in\mathcal{B}(\xi_k, r_{\xi})$ then the mapping $J_k$ defined
by \eqref{eq:inverse_operator} is single-valued in a neighbourhood $\mathcal{B}(\bar{v}^k, r_v)$, where $\bar{v}^k := \tilde{F}'_k\bar{z}^k
- F(\bar{z}^k) - C\xi_k$.
Moreover, the following inequality holds:
\begin{equation}\label{eq:lm32_estimate}
\norm{J_k(v) - J_k(v')} \leq \beta\norm{v-v'}, ~\forall v, v' \in\mathcal{B}(\bar{v}_k, r_v), 
\end{equation}
where $\beta := \frac{\gamma}{1 - \gamma\kappa} > 0$ is a Lipschitz constant.
\end{lemma}

\begin{proof}
Let us fix a neighbourhood $\mathcal{B}(\bar{v}^k, r_v)$ of $\bar{v}^k$.
Suppose for contradiction that $J_k$ is not single-valued in $\mathcal{B}(\bar{v}^k, r_v)$, then for a given $v$ the set $J_k(v)$ contains at
least two points $z$ and $z'$ such that  $\norm{z-z'} \neq 0$. We have 
\begin{equation}\label{eq:s2_proof_lm32_1}
v \in \tilde{F}'_kz + \mathcal{N}_K(z) ~\mathrm{and} ~ v \in \tilde{F}'_kz' + \mathcal{N}_K(z'). 
\end{equation}
Let
\begin{eqnarray}\label{eq:s2_proof_lm32_2}
&&\delta := v - [\tilde{F}'_k\bar{z}^k -  F(\bar{z}^k) - C\xi_k] + [F'(\bar{z}^k) - \tilde{F}'_k](z-\bar{z}^k), \nonumber\\
[-1.5ex]\textrm{and}~~&&\\[-1.5ex]
&&\delta' := v -  [\tilde{F}'_k\bar{z}^k - F(\bar{z}^k) - C\xi_k] + [F'(\bar{z}^k) - \tilde{F}'_k](z'-\bar{z}^k). \nonumber 
\end{eqnarray}
Then \eqref{eq:s2_proof_lm32_1} can be written as
\begin{eqnarray}\label{eq:s2_proof_lm32_3}
&&\delta \in F(\bar{z}^k) + F'(\bar{z}^k)(z-\bar{z}^k) + C\xi_k + \mathcal{N}_K(z), \nonumber\\
[-1.5ex]\textrm{and}~~\\[-1.5ex]
&&\delta' \in F(\bar{z}^k) + F'(\bar{z}^k)(z'-\bar{z}^k) + C\xi_k + \mathcal{N}_K(z'). \nonumber
\end{eqnarray}
Since $v$ in the neighbourhood $\mathcal{B}(\bar{v}^k, r_v)$ of $\bar{v}^k := \tilde{F}'_k\bar{z}^k -  F(\bar{z}^k) - C\xi_k$, we have
\begin{eqnarray}\label{eq:s2_proof_lm32_4}
\norm{\delta} && \leq \norm{v - \bar{v}^k} + \norm{[F'(\bar{z}^k) - \tilde{F}'_k](z-\bar{z}^k)} \nonumber\\
&& \leq r_v + \norm{F'(\bar{z}^k) - \tilde{F}'_k}\norm{z-\bar{z}^k} \nonumber\\
&& \overset{\tiny\eqref{eq:A3_condition}}{\leq} r_v + \kappa\norm{z-\bar{z}^k}. \nonumber 
\end{eqnarray}
From this inequality, we see that we can shrink $\mathcal{B}(\bar{z}^k, r_z)$ and $\mathcal{B}(\bar{v}^k, r_v)$ sufficiently small (if 
necessary) such that $\norm{\delta} \leq \bar{r}_{\delta}$. Hence, $\delta \in \mathcal{B}(0, \bar{r}_{\delta})$.
Similarly, $\delta' \in \mathcal{B}(0, \bar{r}_{\delta})$.

Now, using the strong regularity assumption of $\mathrm{P}(\xi_k)$ at $\bar{z}^k$, it follows from \eqref{eq:s2_proof_lm32_3} that
\begin{eqnarray}\label{eq:s2_proof_lm32_5}
&&\norm{z - z'} \leq \gamma\norm{\delta - \delta'}. 
\end{eqnarray}
However, using \eqref{eq:s2_proof_lm32_2}, we have 
\begin{eqnarray*}
\norm{\delta - \delta'} &&= \norm{[F'(\bar{z}^k) - \tilde{F}'_k](z-z')} \nonumber\\
&& \leq  \norm{F'(\bar{z}^k)-\tilde{F}'_k}\norm{z-z'} \nonumber\\
&&\overset{\tiny\eqref{eq:A3_condition}}{\leq} \kappa\norm{z-z'}. 
\end{eqnarray*}
Plugging this inequality into \eqref{eq:s2_proof_lm32_5} and then using the condition $\gamma\kappa < \frac{1}{2} < 1$, we get
\begin{equation*}
\norm{z - z'}  < \norm{z-z'}, 
\end{equation*}
which contradicts to $z\neq z'$. Hence, $J_k$ is single-valued.

Finally, we prove the Lipschitz continuity of $J_k$. 
Let $z = J_k(v)$ and $z' = J_k(v')$, where $v, v' \in \mathcal{B}(\bar{v}^k, r_v)$. Similar to \eqref{eq:s2_proof_lm32_3}, these expressions
can be written equivalently to
\begin{eqnarray}\label{eq:s2_proof_lm32_6}
&&\delta \in F(\bar{z}^k) + F'(\bar{z}^k)(z-\bar{z}^k) + C\xi_k + \mathcal{N}_K(z), \nonumber\\
[-1.5ex]\mathrm{and}~~\\[-1.5ex]
&&\delta' \in F(\bar{z}^k) + F'(\bar{z}^k)(z' - \bar{z}^k) + C\xi_k + \mathcal{N}_K(z'), \nonumber
\end{eqnarray}
where
\begin{eqnarray}\label{eq:s2_proof_lm32_7}
&&\delta := v - [\tilde{F}'_k\bar{z}^k -  F(\bar{z}^k) - C\xi_k] + [F'(\bar{z}^k) - \tilde{F}'_k](z-\bar{z}^k), \nonumber\\
[-1.5ex]\mathrm{and}~~\\[-1.5ex]
&&\delta' := v' -  [\tilde{F}'_k\bar{z}^k - F(\bar{z}^k) - C\xi_k] + [F'(\bar{z}^k) - \tilde{F}'_k](z' - \bar{z}^k).
\nonumber  
\end{eqnarray}
By using again the strong regularity assumption, it follows from \eqref{eq:s2_proof_lm32_6} and \eqref{eq:s2_proof_lm32_7} that
\begin{eqnarray}\label{eq:s2_proof_lm32_8}
\norm{z - z'} &&\leq \gamma\norm{\delta - \delta'} \nonumber\\
&& \leq \gamma\norm{v-v'} + \gamma\norm{[F'(\bar{z}^k) - \tilde{F}'_k](z-z')} \nonumber\\
&& \overset{\tiny\eqref{eq:A3_condition}}{\leq} \gamma\norm{v-v'} + \gamma\kappa\norm{z-z'}. \nonumber 
\end{eqnarray}
Since $\gamma\kappa < \frac{1}{2} < 1$, rearranging the last inequality we get
\begin{equation*}
\norm{z-z'} \leq \frac{\gamma}{1-\gamma\kappa}\norm{v-v'}, 
\end{equation*}
which shows that $J_k$ satisfies \eqref{eq:lm32_estimate} with a constant $\beta := \frac{\gamma}{1-\gamma\kappa} > 0$. 
\end{proof}

Let us recall that if $z^{k+1}$ is a KKT of the convex subproblem $\mathrm{P}(z^k, A_k, H_k;\xi_{k+1})$ then
\begin{equation*}
0 \in \tilde{F}'_k(z^{k+1} - z^k) + F(z^k) + C\xi_{k+1} + \mathcal{N}_K(z^{k+1}).
\end{equation*}
According to Lemma \ref{le:inverse_operator}, if $z^k \in \mathcal{B}(\bar{z}^k, r_z)$ then problem \ref{eq:convex_subprob} is uniquely 
solvable. We can write its KKT condition equivalently as 
\begin{equation}\label{eq:zkplus1}
z^{k+1} = J_k\left(\tilde{F}'_kz^k - F(z^k) - C\xi_{k+1}\right). 
\end{equation}
Since $\bar{z}^{k+1}$ is the solution of \eqref{eq:GE} at $\xi_{k+1}$, we have $0 = F(\bar{z}^{k+1}) + C\xi_{k+1} + \bar{u}^{k+1}$, where 
$\bar{u}^{k+1} \in \mathcal{N}_K(\bar{z}^{k+1})$. Moreover, since $\bar{z}^{k+1} = J_k(\tilde{F}'_k\bar{z}^{k+1} + \bar{u}^{k+1})$, we can
write
\begin{equation}\label{eq:sol_express}
\bar{z}^{k+1} = J_k\left(\tilde{F}'_k\bar{z}^{k+1} - F(\bar{z}^{k+1}) - C\xi_{k+1}\right). 
\end{equation}
The main result of this section is stated in the following theorem.

\begin{theorem}\label{th:contraction_estimate}
Suppose that Assumptions \aref{as:A1}-\aref{as:A2} are satisfied for some $\xi_0\in\mathcal{P}$.
Then, for $k\geq 0$ and $\bar{z}^k\in Z^{*}(\xi_k)$, if $\mathrm{P}(\xi_k)$ is strongly regular at $\bar{z}^k$ then there exist
neighborhoods $\mathcal{B}(\bar{z}^k, r_z)$ and $\mathcal{B}(\xi_k, r_{\xi})$ such that:
\begin{itemize}
\item[a)] The set of KKT points $Z^{*}(\xi_{k+1})$ of $\mathrm{P}(\xi_{k+1})$ is nonempty for any $\xi_{k+1}\in\mathcal{B}(\xi_k, r_{\xi})$.
\item[b)] If, in addition, Assumption \aref{as:A3}$\mathrm{a)}$ is satisfied then subproblem $\mathrm{P}(z^k,A_k,H_k;\xi_{k+1})$ is uniquely
solvable in
the neighborhood $\mathcal{B}(\bar{z}^k, r_z)$.
\item[c)] Moreover, if, in addition, Assumption \aref{as:A3}$\mathrm{b)}$ is satisfied then the sequence $\{z^k\}_{k\geq 0}$ generated by
Algorithm \ref{alg:A1}, where $\xi_{k+1} \in \mathcal{B}(\xi_k, r_{\xi})$,
guarantees
\begin{eqnarray}\label{eq:contraction_estimate}
\norm{z^{k+1} - \bar{z}^{k+1}} &&\leq \left(\alpha + c_1\norm{z^k - \bar{z}^k}\right)\norm{z^k - \bar{z}^k} \nonumber\\
[-1.5ex]\\[-1.5ex]
&& + \left( c_2 + c_3\norm{\xi_{k+1}-\xi_k}\right)\norm{\xi_{k+1}-\xi_k}, \nonumber
\end{eqnarray}
where $ 0 \leq \alpha < 1$, $0 \leq c_i  < +\infty$, $i=1,\dots, 3$ and $c_2 > 0$ are given constants and $\bar{z}^{k+1}\in Z^{*}(\xi_{k+1})$.
\end{itemize}
\end{theorem}

\begin{proof}
We prove the theorem by induction.
For $k=0$, we have $Z^{*}(\xi_0)$ is nonempty by Assumption \aref{as:A2}. Now, we assume $Z^{*}(\xi_k)$ is nonempty for some $k\geq 0$. We
will prove that $Z^{*}(\xi_{k+1})$ is nonempty for some $\xi_{k+1}\in\mathcal{B}(\xi_k, r_{\xi})$, a neighborhood of $\xi_k$. 
 
Indeed, since $Z^{*}(\xi_k)$ is nonempty for some $\xi_k\in\mathcal{P}$, we take an arbitrary $\bar{z}^k\in Z^{*}(\xi_k)$ such that 
$\mathrm{P}(\xi_k)$ is strong regular at $\bar{z}^k$. Now, by applying Lemma \ref{le:nonemptyness_of_Gamma} to problem $\mathrm{P}(\xi_k)$,
then we conclude that there exist neighborhoods $\mathcal{B}(\bar{z}^k, r_z)$ of $\bar{z}^k$ and $\mathcal{B}(\xi_k, r_{\xi})$ of $\xi_k$
such that $Z^{*}(\xi_{k+1})$ is nonempty for any $\xi_{k+1}\in\mathcal{B}(\xi_k, r_{\xi})$.

Next, if, in addition, Assumption \aref{as:A3}$\mathrm{a)}$ holds then the conclusions of Lemma \ref{le:inverse_operator} hold. By
induction, we conclude that the convex subproblem $\mathrm{P}(\bar{z}^k,A_k,\xi_k)$ is uniquely solvable in $\mathcal{B}(\bar{z}^k, r_z)$
for any $\xi_{k+1}\in\mathcal{B}(\xi_k, r_{\xi})$.
   
Finally, we prove inequality \eqref{eq:contraction_estimate}.
From \eqref{eq:zkplus1}, \eqref{eq:sol_express} and the mean-value theorem and Assumption \aref{as:A3}$\mathrm{b)}$, we have
\begin{eqnarray}\label{eq:proof_thm31_est1}
\norm{z^{k\!+\!1} \!\!-\! \bar{z}^{k\!+\!1}} && \overset{\tiny\eqref{eq:zkplus1}}{=} \norm{J_k\left((\tilde{F}'_kz^k - F(z^k) -
C\xi_{k+1}\right) - \bar{z}^{k+1}} \nonumber\\
&& \overset{\tiny\eqref{eq:sol_express}}{=} \norm{J_k\left(\tilde{F}'_kz^k - F(z^k) - C\xi_{k+1}\right) - J_k\left(\tilde{F}'_k\bar{z}^{k+1}
-  F(\bar{z}^{k+1}) - C\xi_{k+1}\right)}
\nonumber \\
&& \overset{\tiny\eqref{eq:lm32_estimate}}{\leq} \beta\norm{\tilde{F}'_k(z^k - \bar{z}^{k+1}) - F(z^k) + F(\bar{z}^{k+1})} \nonumber\\
&& = \beta\!\norm{\left[\tilde{F}'_k(z^k \!-\! \bar{z}^k) \!-\! F(z^k) \!+\! F(\bar{z}^k)\right] \!+\! \left[F(\bar{z}^{k \!+\! 1})
\!-\! F(\bar{z}^k) \!-\! \tilde{F}_k'(\bar{z}^{k\!+\!1} \!-\! \bar{z}^k)\right]} \nonumber\\
[-1.5ex]\\[-1.5ex]
&&  \leq \beta\!\norm{[\tilde{F}_k' - F'(\bar{z}^k)](z^k-\bar{z}^k) - \int_0^1\!\![F'(\bar{z}^k \!+\! t(z^k \!-\! \bar{z}^k)) \!-\!
F'(\bar{z}^k)](z^k \!-\! \bar{z}^k)dt} \nonumber\\
&& +  \beta\!\norm{[\tilde{F}_k' \!-\! F'(\bar{z}^k)](\bar{z}^{k \!+\!1} \!\!-\! \bar{z}^k) - \int_0^1\!\!\! [F'(\bar{z}^k \!+\!
t(\bar{z}^{k \!+\! 1} \!\!\!\!-\! \bar{z}^k)) \!-\! F'(\bar{z}^k)](z^{k \!+\! 1} \!\!\!\!-\! \bar{z}^k)dt} \nonumber\\
&& \overset{\tiny\eqref{eq:A3_condition}+\eqref{eq:A4_condition}}{\leq} \beta\left(\kappa + \frac{\omega}{2}\norm{z^k -
\bar{z}^k}\right)\norm{z^k-\bar{z}^k} \nonumber\\
&& + \beta\left(\kappa + \frac{\omega}{2}\norm{\bar{z}^{k+1} - \bar{z}^k}\right)\norm{\bar{z}^{k+1} - \bar{z}^k} \nonumber. 
\end{eqnarray}
By substituting \eqref{eq:lm31_estimate} into \eqref{eq:proof_thm31_est1} we obtain
\begin{eqnarray*}
\norm{z^{k+1} - \bar{z}^{k+1}} &&\leq \beta\left(\kappa + \frac{\omega}{2}\norm{z^k - \bar{z}^k}\right)\norm{z^k-\bar{z}^k} \nonumber\\
&& + \beta\left(\kappa\bar{\sigma} + \frac{\omega\bar{\sigma}^2}{2}\norm{\xi_{k+1} - \xi_k}\right)\norm{\xi_{k+1} - \xi_k}.  
\end{eqnarray*}
If we define $\alpha := \beta\kappa = \frac{\gamma\kappa}{1-\gamma\kappa} < 1$ due to \aref{as:A3}a), $c_1 :=
\frac{\gamma\omega}{2(1-\gamma\kappa)}
\geq 0$, $c_2 :=  \frac{\gamma\kappa\bar{\sigma}}{1-\gamma\kappa} > 0$ and $c_3 := \frac{\gamma\omega\bar{\sigma}^2}{2(1-\gamma\kappa)} \geq 0$ as
four given constants then the last inequality is indeed \eqref{eq:contraction_estimate}.
\end{proof}

The following corollary shows the stability of the approximate sequence $\{z^k\}_{k\geq 0}$ generated by Algorithm \ref{alg:A1}. 
   
\begin{corollary}\label{co:tracking_corollary}
Under the assumptions of Theorem \ref{th:contraction_estimate}, there exists a positive number $0 < r_z < \bar{r}_z := (1-\alpha)c_1^{-1}$
such that if the initial point $z^0$ in Algorithm \ref{alg:A1} is chosen such that $\norm{z^0 - \bar{z}^0}\leq r_z$, where $\bar{z}^0 \in
Z^{*}(\xi_0)$ then, for any $k\geq 0$, we have
\begin{equation}\label{eq:co31_stability}
\norm{z^{k+1} - \bar{z}^{k+1}} \leq r_z, 
\end{equation}
provided that $\norm{\xi_{k+1} - \xi_k} \leq r_{\xi}$, where $\bar{z}^{k+1} \in Z^{*}(\xi_{k+1})$ and $0 < r_{\xi} \leq \bar{r}_{\xi}$ with
\begin{equation*}
\bar{r}_{\xi} := \begin{cases} (2c_3)^{-1}\left[\sqrt{c_2^2 + 4c_3r_z(1-\alpha - c_1r_z)} - c_2\right] & \textrm{if} ~ c_3 > 0,\\
c_2^{-1}r_z(1-\alpha - c_1r_z) &\textrm{if}~ c_3 = 0.                  
\end{cases}
\end{equation*}
Consequently, the error sequence $\{\textrm{e}_k\}_{k\geq 0}$, where $\textrm{e}_{k} := \norm{z^{k} - \bar{z}^{k}}$, between the exact KKT
point $\bar{z}^{k}$ and the approximate KKT point $z^{k}$ of $\mathrm{P}(\xi_{k})$ is bounded.   
\end{corollary}

\begin{proof}
Since $0 \leq \alpha < 1$, we have $\bar{r}_z := (1-\alpha)c_1^{-1} > 0$. Let us choose $r_z$ such that $0 < r_z < \bar{r}_z$.
If $z^0\in \mathcal{B}(\bar{z}^0, r_z)$, i.e. $\norm{z^0 - \bar{z}^0} \leq r_z$, then it follows from
\eqref{eq:contraction_estimate} that 
\begin{eqnarray*}
\norm{z^1 - \bar{z}^1} \leq (\alpha + c_1r_z)r_z + (c_2 + c_3\norm{\xi_1 - \xi_0})\norm{\xi_1 - \xi_0}. 
\end{eqnarray*}
In order to ensure $\norm{z^1 - \bar{z}^1} \leq r_z$, we need $(c_2 + c_3\norm{\xi_1 - \xi_0})\norm{\xi_1 - \xi_0} \leq \rho := (1-\alpha - c_1r_z)r_z$.
Since $0 < r_z < \bar{r}_z$, $\rho > 0$. The last condition leads to $\norm{\xi_1 - \xi_0} \leq (2c_3)^{-1}(\sqrt{c_2^2 + 4c_3\rho} - c_2)$ if $c_3 > 0$ and
$\norm{\xi_1 - \xi_0} \leq c_2^{-1}r_z(1-\alpha - c_1r_z)$ if $c_3 = 0$.
By induction, we conclude that inequality \eqref{eq:co31_stability} holds for all $k\geq 0$.
\end{proof}

The conclusion of Corollary \ref{co:tracking_corollary} is illustrated in Figure \ref{fig:for_proof}, where the approximate sequence
$\{z^k\}_{k\geq 0}$ computed by 
Algorithm \ref{alg:A1} remains close to the sequence of the true KKT points $\{\bar{z}^k\}_{k\geq 0}$ if the starting point $z^0$ is
sufficiently close to $\bar{z}_0$.

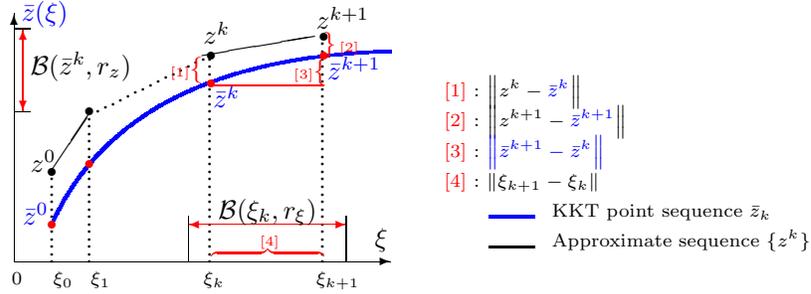
\begin{figure}[ht]
\setlength{\unitlength}{1mm}
\begin{picture}(60, 40)\label{fig:proof}
  \linethickness{0.075mm}
  \put(20, 0){\vector(0, 1){33}}
  \put(21,32){\color{blue}$\bar{z}(\xi)$}
  \put(20, 0){\vector(1, 0){50}} 
  \put(68, 2){$\xi$}
  {\color{blue}
  \linethickness{0.3mm}
  \qbezier(25, 5)(35, 27)(70, 28)
  }
  \linethickness{0.075mm}
  \put(25,12.2){\line(2,3){5}}
  \multiput(25,0)(0,1){13}{\circle*{0.01}}
  \multiput(30,0)(0,1){20}{\circle*{0.01}}
  \put(45,27.5){\line(6,1){15}}
  \multiput(46,0)(0,1){28}{\circle*{0.01}}
  \multiput(61,0)(0,1){31}{\circle*{0.01}}
  \put(25,12){\circle*{1}}
  {\color{red}\put(25,5){\circle*{1}}}
  \put(30,20){\circle*{1}}
  {\color{red}\put(30,13){\circle*{1}}}
  {\color{red}\put(45,23.7){\circle*{1}}}
  \put(45,27.5){\circle*{1}}
  \put(60,30){\circle*{1}}
  {\color{red}\put(60,27.5){\circle*{1}}}
  \multiput(30,20)(1,0.5){16}{\circle*{0.1}} 
  \put(24,-3){{\scriptsize$\xi_0$}}
  \put(29,-3){{\scriptsize$\xi_1$}}
  \put(44,-3){{\scriptsize$\xi_k$}}
  \put(59,-3){{\scriptsize$\xi_{k+1}$}}
  \put(21, 12){$z^0$}
  \put(20, 5){$\color{blue}\bar{z}^0$}
  \put(44, 29){$z^k$}
  \put(45.3, 20.5){\color{blue}$\bar{z}^k$}
  \put(59, 31){$z^{k+1}$}
  \put(60.5, 24.3){$\color{blue}\bar{z}^{k+1}$}
  \put(18.5,-3){\scriptsize$0$}
  \put(42.0,25.5){\color{red}\makebox(0,0){{\tiny[1]}$\left\{\right.$}}
  \put(62.3,28.8){\color{red}\makebox(0,0){$\left\}\right.${\tiny[2]}}}
  \put(45,0){\color{red}\tiny$\overbrace{\rule{15mm}{0cm}}$}
  \put(51.5,2.2){\color{red}\tiny[4]}
  \put(45,23.5){\color{red}\line(1,0){15}}
  \put(58.4,25.3){\color{red}\makebox(0,0){{\tiny[3]}$\left\{\right.$}}
  \put(20,20){\color{red}\vector(0,1){11}}
  \put(20,31){\color{red}\vector(0,-1){11}}
  \put(19,31){\line(1,0){2}}    
  \put(19,20){\line(1,0){2}}
  \put(21,25){$\mathcal{B}(\bar{z}^k, r_z)$}
  \put(42,5){\color{red}\vector(1,0){21}}
  \put(63,5){\color{red}\vector(-1,0){21}}
  \put(42,0){\line(0,1){6}}
  \put(63,0){\line(0,1){6}}
  \put(46,5.8){$\mathcal{B}(\xi_k,r_{\xi})$}
  \put(76,22){\scriptsize{\color{red}[1]}~:~$\norm{z^k-{\color{blue}\bar{z}^k}}$}  
  \put(76,18){\scriptsize{\color{red}[2]}~:~$\norm{z^{k+1}-{\color{blue}\bar{z}^{k+1}}}$}
  \put(76,14){\scriptsize{\color{red}[3]}~:~$\color{blue}\norm{\bar{z}^{k+1}-\bar{z}^{k}}$}
  \put(76,10){\scriptsize{\color{red}[4]}~:~$\norm{\xi_{k+1}-\xi_k}$}
  \put(76,6){\scriptsize{\color{red}}~~}\put(82,6){{\color{blue}\linethickness{0.3mm}\line(1,0){6}}~~\textrm{\scriptsize{KKT point sequence 
   $\bar{z}_k$}}}
  \put(76,2){\scriptsize{\color{red}}~~}\put(82,2){{\color{black}\linethickness{0.3mm}\line(1,0){6}}~~\textrm{\scriptsize{Approximate
sequence $\{z^k\}$}}}
\end{picture}
\vskip 0.2cm
\caption{The approximate sequence $\{z^k\}_{k\geq 0}$ along the trajectory $\bar{z}(\cdot)$ of the KKT points.}\label{fig:for_proof}   
\end{figure}  

\subsection{A contraction estimate for APCSCP using an exact Jacobian matrix}
If $A_k \equiv g'(x^k)$ then the correction vector $m^k = 0$ and the convex subproblem \ref{eq:convex_subprob} collapses to the following
one:
\begin{equation}\label{eq:exact_convex_subprob}
\makeatletter
\def\tagform@#1{\maketag@@@{#1\@@italiccorr}}
\makeatother
\left\{\begin{array}{cl}
\displaystyle
\min_{x\in\R^n} &\Big\{ c^Tx + \frac{1}{2}(x-x^k)^TH_k(x-x^k) \Big\} \\
\textrm{s.t.} &g(x^k) + g'(x^k)(x-x^k) + M\xi = 0, \\
       &x \in \Omega.
\end{array}\right.\tag{$\mathrm{P}(x^k, H_k;\xi)$}
\end{equation}
Note that problem \ref{eq:exact_convex_subprob} does not depend on the multiplier $y^k$ if we choose $H_k$ independently of $y^k$. We refer to a variant of
Algorithm \ref{alg:A1}
where we use the convex subproblem \ref{eq:exact_convex_subprob} instead of \ref{eq:convex_subprob} as a \textit{predictor-corrector SCP
algorithm} (PCSCP) for solving a sequence of the optimization problems $\{\mathrm{P}(\xi_k)\}_{k\geq 0}$. 

Instead of Assumption \aref{as:A3}a) in the previous section, we make the following assumption.
\setcounter{assumptionm}{2}
\begin{assumptionm}\label{as:A3'}
There exists a constant $0 \leq \tilde{\kappa} < \frac{1}{2\gamma}$ such that 
\begin{equation}\label{eq:A3'_assumption}
\norm{\nabla^2_x\mathcal{L}(\bar{z}^k) - H_k} \leq \tilde{\kappa}, ~\forall k\geq0.
\end{equation}
where $\nabla^2_x\mathcal{L}(z)$ defined by \eqref{eq:P_E11_term}.
\end{assumptionm}

Assumption \aref{as:A3'} requires that the approximation $H_k$ to the Hessian matrix $\nabla^2_x\mathcal{L}(\bar{z}^k)$ of the Lagrange
function $\mathcal{L}$ at $\bar{z}^k$ is sufficiently close. Note that matrix $H_k$ in the framework of the SSDP method in \cite{Correa2002} is not
necessarily positive definite.

\vskip0.1cm
\begin{example}\label{ex:example3} 
Let us continue analyzing example \eqref{eq:tutorial_example}.
The Hessian matrix of the Lagrange function $\mathcal{L}$ associated with the equality constraint $x_1^2 + 2x_2 + 2 - 4\xi = 0$ is
$\nabla_x^2\mathcal{L}(x^{*}_{\xi},y^{*}_1) = \left[\begin{smallmatrix}2y^{*}_1 & 0\\ 0 & 0\end{smallmatrix}\right]$, where $y^{*}_1$ is the multiplier
associated with the equality constraint at $x^{*}_{\xi}$.
Let us choose a positive semidefinite matrix $H_k := \left[\begin{smallmatrix}h_{11} & 0 \\ 0 & 0\end{smallmatrix}\right]$, where $h_{11} \geq 0$, then
$\norm{\nabla^2_x\mathcal{L}(x^{*}_{\xi}, y^{*}_1) - H_k} = \abs{y^{*}_1-h_{11}}$.
Since $y_1^{*} \geq 0$, for an arbitrary $\tilde{\kappa} > 0$, we can choose $h_{11}\geq 0$ such that $\abs{h_{11}-y^{*}_1} \leq \tilde{\kappa}$.
Consequently, the condition \eqref{eq:A3'_assumption} is satisfied. In the example \eqref{eq:tutorial_example} of Subsection \ref{subsec:example}, we choose
$h_{11} = 0$.
\end{example} 

The following theorem shows the same conclusions as in Theorem \ref{th:contraction_estimate} and Corollary \ref{co:tracking_corollary} for
the \textit{predictor-corrector SCP algorithm}. 

\begin{theorem}\label{th:exact_jacobian_estimate}
Suppose that Assumptions \aref{as:A1}-\aref{as:A2} are satisfied for some $\xi_0\in\mathcal{P}$.
Then, for $k\geq 0$ and $\bar{z}^k\in Z^{*}(\xi_k)$, if $\mathrm{P}(\xi_k)$ is strongly regular at $\bar{z}^k$ then there exist
neighborhoods $\mathcal{B}(\bar{z}^k, r_z)$ and $\mathcal{B}(\xi_k, r_{\xi})$ such that:
\begin{itemize}
\item[a)] The set of KKT points $Z^{*}(\xi_{k+1})$ of $\mathrm{P}(\xi_{k+1})$ is nonempty for any $\xi_{k+1}\in\mathcal{B}(\xi_k, r_{\xi})$.

\item[b)] If, in addition, Assumption \aref{as:A3'} is satisfied then subproblem $\mathrm{P}(x^k, H_k;\xi_{k+1})$ is uniquely
solvable in the neighborhood $\mathcal{B}(\bar{z}^k, r_z)$.

\item[c)] Moreover, if, in addition, Assumption \aref{as:A3}$\mathrm{b)}$ then the sequence $\{z^k\}_{k\geq 0}$ generated by the PCSCP,
where $\xi_{k+1} \in \mathcal{B}(\xi_k, r_{\xi})$, guarantees
the following inequality:
\begin{eqnarray}\label{eq:contraction_estimate1}
\norm{z^{k+1} - \bar{z}^{k+1}} &&\leq \left(\tilde{\alpha} + \tilde{c}_1\norm{z^k-\bar{z}^k}\right)\norm{z^k - \bar{z}^k} \nonumber\\
&& + \left(\tilde{c}_2 + \tilde{c}_3\norm{\xi_{k+1}-\xi_k}\right)\norm{\xi_{k+1}-\xi_k},
\end{eqnarray}
where $ 0 \leq \tilde{\alpha} < 1$, $0 \leq \tilde{c}_i < +\infty$, $i=1,\cdots,3$ and $\tilde{c}_2 > 0$ are given constants and $\bar{z}^{k+1} \in
Z^{*}(\xi_{k+1})$. 

\item[d)] If the initial point $z^0$ in the \textit{PCSCP} is chosen such that $\norm{z^0 - \bar{z}^0}\leq \tilde{r}_z$, where $\bar{z}^0
\in Z^{*}(\xi_0)$ and $0 < \tilde{r}_z < \tilde{\bar{r}}_z := \tilde{c}_1^{-1}(1 - \tilde{\alpha})$, then:
\begin{equation}\label{eq:co31_stability1}
\norm{z^{k+1} - \bar{z}^{k+1}} \leq \tilde{r}_z, 
\end{equation}
provided that $\norm{\xi_{k+1}-\xi_k} \leq \tilde{r}_{\xi}$ with $0 < \tilde{r}_{\xi} \leq \bar{\tilde{r}}_{\xi}$,
\begin{equation*}
\bar{\tilde{r}}_{\xi} := \begin{cases} (2\tilde{c}_3)^{-1}\left[\sqrt{{\tilde{c}_2}^2 + 4\tilde{c}_3\tilde{r}_z(1 -\tilde{\alpha} -
\tilde{c}_1\tilde{r}_z)} - \tilde{c}_2\right] &\textrm{if}~ \tilde{c}_3 > 0,\\
\tilde{c}_2^{-1}\tilde{r}_z(1 - \tilde{\alpha} - \tilde{c}_1\tilde{r}_z) &\textrm{if} ~ \tilde{c}_3 = 0.
\end{cases} 
\end{equation*}
Consequently, the error sequence  $\{ \norm{z^{k} - \bar{z}^{k}} \}_{k\geq 0}$ between the exact KKT point $\bar{z}^{k}$ and
the approximation KKT point $z^{k}$ of $\mathrm{P}(\xi_{k})$ is still bounded. 
\end{itemize}
\end{theorem}

\begin{proof}
The statement a) of Theorem \ref{th:exact_jacobian_estimate} follows from Theorem \ref{th:contraction_estimate}. We prove b).
Since $A_k \equiv g'(x^k)$, the matrix $\tilde{F}'_k$ defined in \eqref{eq:H_matrix} becomes
\begin{equation*}
\hat{\tilde{F}}'_k := \begin{bmatrix}H_k & g'(x^k) \\ g'(x^k) & 0\end{bmatrix}, 
\end{equation*}
Moreover, since $g$ is twice differentiable due to Assumption \aref{as:A1}, $g'$ is Lipschitz continuous with a Lipschitz constant $L_g
\geq 0$ in
$\mathcal{B}(\bar{x}^k,r_z)$. 
Therefore, by Assumption \aref{as:A3'}, we have
\begin{eqnarray}\label{eq:proof_thm51_est1}
\norm{F'(\bar{z}^k) - \hat{\tilde{F}}'_k}^2 &&=  \norm{\begin{bmatrix}\nabla^2_x\mathcal{L}(\bar{z}^k) & g'(\bar{x}^k)^T - g'(x^k)^T \\
g'(\bar{x}^k) -
g'(x^k) & 0\end{bmatrix}}^2 \nonumber\\
&& \leq \norm{\nabla^2_x\mathcal{L}(\bar{z}^k) - H_k}^2 + 2\norm{g'(x^k) - g'(\bar{x}^k)}^2 \\
&& \leq \tilde{\kappa}^2 + 2L_g^2\norm{x^k-\bar{x}^k}^2\nonumber.
\end{eqnarray}
Since $\tilde{\kappa}\gamma < \frac{1}{2}$, we can shrink $\mathcal{B}(\bar{z}^k,r_z)$ sufficiently small such that 
\begin{equation*}
\gamma\sqrt{\tilde{\kappa}^2 + 2L_g^2r^2_z} < \frac{1}{2}. 
\end{equation*}
If we define $\tilde{\kappa}_1 := \sqrt{\tilde{\kappa}^2 + 2L_g^2r^2_z} \geq 0$ then the last inequality and \eqref{eq:proof_thm51_est1}
imply
\begin{equation}\label{eq:proof_A1_est1}
\norm{F'(\bar{z}^k) - \hat{\tilde{F}}'_k} \leq \tilde{\kappa}_1, 
\end{equation}
where $\tilde{\kappa}_1\gamma < \frac{1}{2}$.
Similar to the proof of Lemma \ref{le:inverse_operator}, we can show that the mapping $\hat{J}_k := (\hat{\tilde{F}}'_k +
\mathcal{N}_K)^{-1}$ is single-valued and Lipschitz continuous with a Lipschitz constant $\tilde{\beta} :=
\gamma(1-\gamma\tilde{\kappa}_1)^{-1} > 0$ in $\mathcal{B}(\bar{z}^k, r_z)$. Consequently, the convex problem $\mathrm{P}(x^k, H_k;
\xi_{k+1})$ is uniquely solvable in $\mathcal{B}(\bar{z}^k, r_z)$ for all $\xi_{k+1}\in\mathcal{B}(\xi_k, r_{\xi})$, which proves b).
 
With the same argument as the proof of Theorem \ref{th:contraction_estimate}, we can also prove the following estimate
\begin{equation*}
\norm{z^{k+1}-\bar{z}^{k}} \leq \left(\tilde{\alpha}_k + \tilde{c}_1\norm{z^k-\bar{z}^k}\right)\norm{z^k-\bar{z}^k} +
\left(\tilde{c}_2 + \tilde{c}_3\norm{\xi_{k+1}-\xi_k}\right)\norm{\xi_{k+1}-\xi_k}, 
\end{equation*}
where $\tilde{\alpha} := \gamma\tilde{\kappa}_1(1 - \gamma\tilde{\kappa}_1)^{-1} \in [0, 1)$, $\tilde{c}_1 :=
\gamma\omega(2-2\gamma\tilde{\kappa}_1)^{-1} \geq 0$, $\tilde{c}_2 := \gamma\tilde{\kappa}_1\bar{\sigma}(1-1\gamma\tilde{\kappa}_1)^{-1}
> 0$ and $\tilde{c}_3 := \gamma\omega\bar{\sigma}^2(2-2\gamma\tilde{\kappa}_1)^{-1} \geq 0$.
The remaining statements of Theorem \ref{th:exact_jacobian_estimate} are proved similarly to the proofs of Theorem
\ref{th:contraction_estimate} and Corollary \ref{co:tracking_corollary}.
\end{proof}

\paragraph{Remark on updating matrices $A_k$ and $H_k$}
In the adjoint-based predictor-corrector SCP algorithm, an approximate matrix $A_k$ of $g'(x^k)$ and a vector $m^k = (g'(x^k) -
A_k)^Ty^k$ are required at each iteration such that they maintain Assumption \aref{as:A3}. 
Suppose that the initial approximation $A_0$ is known. For given $z^k$ and $A_k$, $k\geq 0$,  we need to compute $A_{k+1}$ and $m^{k+1}$ in an efficient way.
If $\norm{A_k- g'(\bar{x}^{k+1})}$ is still small then we can even use the same matrix $A_k$ for the next iteration, i.e.
$A_{k+1} = A_k$ due to Assumption \aref{as:A3} (see Section \ref{se:examples}).  
Otherwise, matrix $A_{k+1}$ can be constructed in different ways, e.g. by using low-rank updates or by a low accuracy computation.
As by an inexactness computation, we can either use the two sided rank-1 updates (TR1) \cite{Diehl2010, Griewank2002} or the Broyden formulas \cite{Schlenkrich2010}.
However, it is important to note that the use of the low-rank update for matrix $A_k$ might destroy possible sparsity structure of matrix $A_k$.
Then high-rank updates might be an option \cite{Bock1984,Griewank1982}.

In Algorithm \ref{alg:A1} we can set matrix $H_k = 0$ for all $k\geq 0$. However, this matrix can be updated at each iteration by using
BFGS-type formulas or the projection of $\nabla^2_x\mathcal{L}(z^k)$ onto $\mathcal{S}^n_{+}$.

\section{Applications in nonlinear programming}\label{sec:SCP_case}
If the set of parameters $\Sigma$ collapses to one point, i.e. $\Sigma := \{\xi\}$ then, without loss of generality, we assume that $\xi=0$
and problem \ref{eq:param_prob} is reduced to a nonlinear programming problem of the form: 
\begin{equation}\label{eq:nlp_prob}
\left\{\begin{array}{cl}
\displaystyle\min_{x\in\R^n} & f(x) := c^Tx \\
\textrm{s.t.}  & g(x) = 0,\\
& x\in\Omega,
\end{array}\right.\tag{$\mathrm{P}$}
\end{equation}
where $c$, $g$ and $\Omega$ are as in \ref{eq:param_prob}. 
In this section we develop local optimization algorithms for solving \eqref{eq:nlp_prob}. 

The KKT condition for problem \eqref{eq:nlp_prob} is expressed as:
\begin{equation}\label{eq:nlp_KKT}
\begin{cases}0 \in c + g'(x)^Ty + \mathcal{N}_{\Omega}(x),\\ 0 = g(x),\end{cases} 
\end{equation}
where $\mathcal{N}_{\Omega}(x)$ defined by \eqref{eq:P_normal_cone}. 
A pair $\hat{z}:=(\hat{x}^T,\hat{y}^T)^T$ satisfying \eqref{eq:nlp_KKT} is called a KKT point, $\hat{x}$ is called a stationary point
and $\hat{y}$ is the corresponding multiplier of \eqref{eq:nlp_prob}, respectively. We denote by $\hat{Z}^{*}$ the set of the KKT points and
by $\hat{S}^{*}$ the set of stationary points of \eqref{eq:nlp_prob}.  

Now, with the mapping $F$ defined as \eqref{eq:P_varphi}, the KKT condition \eqref{eq:nlp_KKT} can be reformulated as a generalized
equation:
\begin{equation}\label{eq:GE}
0 \in F(z) + \mathcal{N}_K(z), 
\end{equation}
where $K = \Omega\times \R^m$ as before and $\mathcal{N}_K(z)$ is the normal cone of $K$ at $z$.  

The subproblem \ref{eq:convex_subprob} in Algorithm \ref{alg:A1} is reduced to 
\begin{equation}\label{eq:convex_subprob0}
\makeatletter
\def\tagform@#1{\maketag@@@{#1\@@italiccorr}}
\makeatother
\left\{\begin{array}{cl}
\displaystyle\min_{x\in\R^n} & c^Tx + (m^j)^T(x-x^j) + \frac{1}{2}(x-x^j)^TH_j(x-x^j)\\
\textrm{s.t.}  & g(x^j) + A_j(x-x^j) = 0,\\
& x\in\Omega.
\end{array}\right.\tag{$\mathrm{P}(z^j, A_j, H_j)$}
\end{equation}
Here, we use the index $j$ in the algorithms for the nonparametric problems (see below) to distinguish from the index $k$ in the parametric cases.

In order to adapt to the theory in the previous sections, we only consider the full-step algorithm for solving \eqref{eq:nlp_prob} which is
called \textit{full-step adjoint-based sequential convex programming} is described as follows.

\noindent\rule[1pt]{\textwidth}{1.0pt}{~~}
\begin{algorithm}\label{alg:A2} 
$\mathrm{(}$\textit{Full-step adjoint-based SCP algorithm}~$\mathrm{(FASCP))}$
\end{algorithm}
\vskip -0.15cm
\noindent\rule[1pt]{\textwidth}{0.5pt}
\vskip -0.15cm
\begin{romannum}
\item[\textbf{Initialization.}] Find an initial guess $x^0\in\Omega$ and $y^0\in\R^m$, a matrix $A_0$ approximated to $g'(x^0)$ and
$H_0\in\mathcal{S}^n_{+}$.  
Set $m^0 := (g'(x^0)-A_0)^Ty^0$ and $j:=0$.
\item[\textbf{Iteration $j$.}] For a given $(z^j, A_j, H_j)$, perform the following steps:
\begin{remunerate}
\item[\textit{Step 1.}] Solve the convex subproblem \ref{eq:convex_subprob0} to obtain a solution
$x^{j+1}_t$ and the corresponding multiplier ${y}^{j+1}$. 

\item[\textit{Step 2.}] If $\norm{x^{j+1}_t-x^j}\leq\varepsilon$, for a given tolerance $\varepsilon>0$, then: terminate. Otherwise, 
compute the search direction $\Delta x^j := x^{j+1}_t-x^j$.

\item[\textit{Step 3.}] Update $x^{j+1} := x^j + \Delta x^j$. 
Evaluate the function value $g(x^{j+1})$, update (or recompute) matrices $A_{j+1}$ and $H_{j+1}\in\mathcal{S}^n_{+}$ (if necessary) and the
correction vector $m^{j+1}$.
Increase $j$ by $1$ and go back to Step 1.
\end{remunerate}
\end{romannum}
\noindent\rule[1pt]{\textwidth}{1pt}
The following corollary shows that the \textit{full-step adjoint-based SCP algorithm} generates an iterative sequence that converges
linearly to a KKT point of \eqref{eq:nlp_prob}. 

\begin{corollary}\label{co:SCP_case}
Let $\hat{Z}^{*}\neq\emptyset$ and $\hat{z}^{*}\in\hat{Z}^{*}$. Suppose that Assumption \aref{as:A1} holds and that problem
\eqref{eq:nlp_prob} is strongly regular at $\hat{z}^{*}$ (in the sense of Definition \ref{de:strong_regularity}). Suppose further that
Assumption \aref{as:A3}$\mathrm{a)}$ is satisfied in $\mathcal{B}(\hat{z}^{*},\hat{r}_z)$.
Then there exists a neighborhood $\mathcal{B}(\hat{z}^{*}, r_z)$ of $\hat{z}^{*}$ such that, in this neighborhood, the convex subproblem
$\mathrm{P}(x^j,A_j, H_j)$ has a unique KKT point $z^{j+1}$ for any $z^j\in\mathcal{B}(\hat{z}^{*}, r_z)$.
Moreover, if, in addition, Assumption \aref{as:A3}$\mathrm{b)}$ holds then the sequence $\{z^j\}_{j\geq 0}$ generated by Algorithm
\ref{alg:A2} starting from $z^0\in \mathcal{B}(\hat{z}^{*}, r_z)$
satisfies 
\begin{equation}\label{eq:adjSCP_est}
\norm{z^{j+1} - \hat{z}^{*}} \leq (\hat{\alpha} + \hat{c}_1\norm{z^j - \hat{z}^{*}})\norm{z^j - \hat{z}^{*}}, ~\forall j\geq 0, 
\end{equation}
where $0 \leq \hat{\alpha} < 1$ and $0 \leq \hat{c}_1  < +\infty$ are given constants. 
Consequently, this sequence converges linearly to $\hat{z}^{*}$, the unique KKT point of \eqref{eq:nlp_prob} in $\mathcal{B}(\hat{z}^{*},
r_z)$. 
\end{corollary}

\begin{proof}
The estimate \eqref{eq:adjSCP_est} follows directly from Theorem \ref{th:contraction_estimate} by taking $\xi_k=0$ for all $k$. The
remaining statement is a consequence of \eqref{eq:adjSCP_est}. 
\end{proof}

If $A_j = g'(x^j)$ then the convex subproblem \ref{eq:convex_subprob0} in Algorithm \ref{alg:A2} is reduced to:
\begin{equation}\label{eq:exact_convex_subprob0}
\makeatletter
\def\tagform@#1{\maketag@@@{#1\@@italiccorr}}
\makeatother
\left\{\begin{array}{cl}
\displaystyle
\min_{x\in\R^n} &\left\{ c^Tx + \frac{1}{2}(x-x^j)^TH_j(x-x^j)\right\}\\
\textrm{s.t.} & g(x^j) + g'(x^j)(x-x^j) = 0, \\
       &x \in \Omega.
\end{array}\right.\tag{$\mathrm{P}(x^j, H_j)$}
\end{equation}
The local convergence of the \textit{full-step SCP algorithm} considered in \cite{Quoc2010} follows from Theorem \ref{th:exact_jacobian_estimate} as
a consequence, which is restated in the following corollary. 

\begin{corollary}\label{co:exact_jacobian_SCP_estimate}
Suppose that Assumption \aref{as:A1} holds and  problem \eqref{eq:nlp_prob} is strongly regular at a KKT point $\hat{z}^{*}\in\hat{Z}^{*}$
(in the sense of Definition \ref{de:strong_regularity}).
Suppose further that Assumptions \aref{as:A3'} and \aref{as:A3}$\mathrm{b)}$ are satisfied.
Then there exists a neighborhood $\mathcal{B}(\hat{z}^{*}, r_z)$ of $\hat{z}^{*}$ such that the \textit{full-step SCP algorithm} starting
from $x^0$ with $(x^0,{y}^0) \in \mathcal{B}(\hat{z}^{*}, r_z)$ generates a sequence $\{z^j\}_{j\geq 0}$ satisfying:
\begin{equation*}\label{eq:exact_SCP_est}
\norm{z^{j+1} - \hat{z}^{*}} \leq (\breve{\alpha} + \breve{c}_1\norm{z^j - \hat{z}^{*}})\norm{z^j - \hat{z}^{*}},
\end{equation*}
where $\breve{\alpha} \in [0,1)$ and $\breve{c}_1 \in [0, +\infty)$ are constants and $z^{j+1}$ is a unique KKT point of the
subproblem $\mathrm{P}(x^j, H_j)$.  
As a consequence, the sequence $\{z^j\}$ converges linearly to $\hat{z}^{*}$, the unique KKT point of \eqref{eq:nlp_prob}
in $\mathcal{B}(\hat{z}^{*}, r_z)$. 
\end{corollary}

Finally, it is necessary to remark that if $\Omega$ is a polyhedral convex set in $\R^n$, i.e. $\Omega$ is the intersection of finitely many
closed half spaces of $\R^n$, then problem \eqref{eq:nlp_prob} also covers the standard nonlinear programming problem.
It was proved in \cite{Dontchev1996} that if $\Omega$ is polyhedral convex and the constraint qualification (LICQ) holds then the strong
regularity concept coincides with the strong second order sufficient condition (SSOSC) for \eqref{eq:nlp_prob}.
In this case, by an appropriate choice of $H_k$, the SCP algorithm collapses to the constrained Gauss-Newton method which has been widely used in numerical
solution of optimal
control problems, see, e.g. \cite{Bock1984}.

\section{Numerical Results}\label{se:examples}
In this section we implement the algorithms proposed in the previous sections to solve the model predictive control problem of a hydro power plant. 

\subsection{Dynamic model} 
We consider a hydro power plant composed of several subsystems connected together. The system includes six dams with turbines $D_i$
($i=1,\dots, 6$) located along a river and three lakes $L_1, L_2$ and $L_3$ as visualized in Fig. \ref{fig:system}. $U_1$ is a
duct connecting lakes $L_1$ and $L2$. $T_1$ and $T_2$ are ducts equipped with turbines and $C_1$ and $C_2$ are ducts equipped with turbines
and pumps. The flows through the turbines and pumps are the controlled variables. The complete model with all the parameters can be found in
\cite{Savorgnan2010}. 
\begin{figure}[!h]\centering{{\tiny
\begin{tikzpicture}[scale=0.3,>=stealth]
\lake{L1}{(23,-3)}{$L_1$};
\lake{L2}{(16,-2)}{$L_2$};
\lake{L3}{(10.5,-9)}{$L_3$};
\coordinate (R1_in) at (28,-7);
\node[right] at (30.5,-6.3) {$q_\mathrm{in}$};
\arrow{30.5,-6.3}{28.5,-6.8};
\reach{R1}{R1_in}{4};
\node[below] at (R1_mb) {$R_1$};
\dam{D1}{R1_out}{$D_1$};
\reach{R2}{D1_out}{4};
\node[below] at (R2_mb) {$R_2$};
\dam{D2}{R2_out}{$D_2$};
\reach{R3}{D2_out}{4};
\node[above] at (R3_ma) {$R_3$};
\dam{D3}{R3_out}{$D_3$};
\reach{R4}{D3_out}{4};
\node[below] at (R4_mb) {$R_4$};
\dam{D4}{R4_out}{$D_4$};
\reach{R5}{D4_out}{4};
\node[below] at (R5_mb) {$R_5$};
\dam{D5}{R5_out}{$D_5$};
\reach{R6}{D5_out}{4};
\node[below] at (R6_mb) {$R_6$};
\dam{D6}{R6_out}{$D_6$};
\reach{R7}{D6_out}{2};
\dduct{R1_ma}{L1_br}{$C_1$};
\uduct{L1_b}{R2_ma}{$T_1$};
\dduct{R4_ma}{L3_br}{$C_2$};
\uduct{L3_bl}{R5_ma}{$T_2$};
\dduct{L1_bl}{L2_br}{$U_1$};
\node[right] at ($(R3_mb)+(4,-0.3)$) {$q_\mathrm{tributary}$};
\arrow{$(R3_mb)+(4,-0.3)$}{$(R3_mb)+(0,0.25)$}
\end{tikzpicture}}
\caption{Overview of the hydro power plant.}\label{fig:system}}
\end{figure}
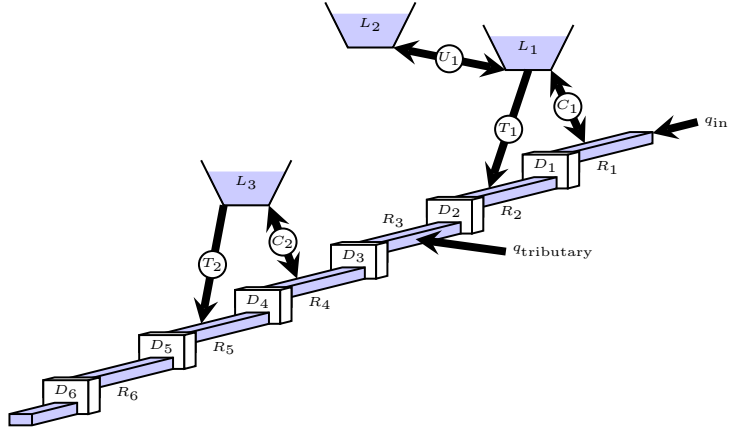

The dynamics of the lakes is given by
\begin{equation}\label{eq:lake_dynamic}
\frac{\partial{h}(t)}{\partial{t}} = \frac{q_{\mathrm{in}}(t) - q_{\mathrm{out}}(t)}{S},   
\end{equation}
where $h(t)$ is the water level and $S$ is the surface area of the lakes; $q_{\mathrm{in}}$ and $q_{\mathrm{out}}$ are the input and output
flows, respectively. The dynamics of the reaches $R_i$ ($i=1,\dots, 6$) is described by the one-dimensional Saint-Venant partial
differential equation:
\begin{equation}\label{eq:Saint_Venant}   
\begin{cases}
&\frac{\partial{q(t,y)}}{\partial{y}} + \frac{\partial{s(t,y)}}{\partial{t}} = 0,\\
&\frac{1}{g}\frac{\partial}{\partial{t}}\left(\frac{q(t,y)}{s(t,y)}\right)
+\frac{1}{2g}\frac{\partial}{\partial{y}}\left(\frac{q^2(t,y)}{s^2(t,y)}\right) +
\frac{\partial{h(t,y)}}{\partial{y}} + I_f(t,y) - I_0(y) = 0. 
\end{cases}
\end{equation}
Here, $y$ is the spatial variable along the flow direction of the river, $q$ is the river flow (or discharge), $s$ is the wetted surface,
$h$ is the water level with respect to the river bed, $g$ is the gravitation acceleration, $I_f$ is the friction slope and $I_0$ is the
river bed slope. The partial differential equation \eqref{eq:Saint_Venant} can be discretized by applying the method of lines in order
to obtain a system of ordinary differential equations. Stacking all the equations together, we represent the dynamics of the system by
\begin{equation}\label{eq:dynamics}
\dot{w}(t) = f(w, u), 
\end{equation}
where the state vector $w\in R^{n_w}$ includes all the flows and the water levels and $u\in R^{n_u}$ represents the input vector.
The dynamic system consists of $n_w = 259$ states and $n_u = 10$ controls. The control inputs are the flows going in the turbines, the ducts and the reaches.

\subsection{Nonlinear MPC formulation}
We are interested in the following NMPC setting:
\begin{equation}\label{eq:NMPC}
\begin{array}{cl}
\displaystyle\min_{w, u} &J(w(\cdot), u(\cdot))\\
\textrm{s.t.} & \dot{w} = f(w, u), ~ w(t) = w_0(t),\\
&u(\tau) \in U, ~ w(\tau) \in W, ~\tau \in [t, t+T]\\
& w(t+T) \in R_T,   
\end{array}
\end{equation}
where the objective function $J(w(\cdot), u(\cdot))$ is given by
\begin{eqnarray}\label{eq:obj_MPC}
J(w(\cdot), u(\cdot)) &:= \int_t^{t+T}\left[(w(\tau) - w_s)^TP(w(\tau) - w_s) +
(u(\tau) - u_s)^TQ(u(\tau) - u_s) \right]d\tau \nonumber\\
[-1.5ex]\\[-1.5ex]
& + (w(t+T) - w_s)^TS(x(t+T) - w_s). \nonumber 
\end{eqnarray}
Here $P, Q$ and $S$ are given symmetric positive definite weighting matrices, and $(w_s, u_s)$ is a steady state of the
dynamics \eqref{eq:dynamics}.
The control variables are bounded by lower and upper bounds, while some state variables are also bounded and the others are unconstrained.
Consequently, $W$ and $U$ are boxes in $\mathbb{R}^{n_w}$ and $\mathbb{R}^{n_u}$, respectively, but $W$ is not necessarily bounded.
The terminal region $R_T$ is a control-invariant ellipsoidal set centered at $w_s$ of radius $r > 0$ and scaling matrix $S$, i.e.:
\begin{equation}\label{eq:terminal_region}
R_T := \left\{ w \in\mathbb{R}^{n_w} ~|~ (w - w_s)^TS(w - w_s) \leq r \right\}. 
\end{equation}
To compute matrix $S$ and the radius $r$ in \eqref{eq:terminal_region} the procedure proposed in \cite{Chen1998} can be used.
In \cite{Jadbabaie2005} it has been shown that the receding horizon control formulation \eqref{eq:NMPC} ensures the stability of the closed-loop system under
mild assumptions.
Therefore, the aim of this example is to track the steady state of the system and to ensure the stability of the system by satisfying the
terminal constraint along the moving horizon.
To have a more realistic simulation we added a disturbance to the input flow $q_{\mathrm{in}}$ at the beginning of the reach $R_1$ and the
tributary flow $q_{\mathrm{tributary}}$.

The matrices $P$ and $Q$ have been set to
\begin{align}\label{eq:PQ}
&P := \textrm{diag}\left(\frac{0.01}{(w_s)_i^2+1}~:~ 1\leq i \leq n_w\right), \nonumber\\
&Q := \textrm{diag}\left(\frac{4}{(u_l + u_b)_i^2 + 1}~:~ 1\leq i \leq n_u\right), \nonumber 
\end{align}
where $u_l$ and $u_b$ is the lower and upper bound of the control input $u$.

\subsection{A short description of the multiple shooting method}
We briefly describe the multiple shooting formulation \cite{Bock1984} which we use to discretize the continuous time problem
\eqref{eq:NMPC}. The time horizon $[t, t + T]$ of $T = 4$ hours is discretized into $H_p = 16$ shooting intervals with $\Delta\tau =
15$ minutes such that $\tau_0 = t$ and $\tau_{i+1} := \tau_i + \Delta\tau$ ($i=0,\dots, H_p-1$). The control $u(\cdot)$ is parametrized by
using a piecewise constant function $u(\tau) = u_i$ for $\tau_i \leq \tau \leq \tau_i+\Delta\tau$ ($i=0,\dots, H_p-1$).

Let us introduce $H_p+1$ shooting node variables $s_i$ ($i=0,\dots, H_p$).
Then, by integrating the dynamic system $\dot{w} = f(w, u)$ in each interval $[\tau_i, \tau_i+\Delta\tau]$, the continuous dynamic
\eqref{eq:dynamics} is transformed into the nonlinear equality constraints of the form:
\begin{equation}\label{eq:eqcon}
g(x) + M\xi:= \begin{bmatrix}s_0 - \xi\\ w(s_0,u_0) - s_1\\ \dots\\ w(s_{H_p-1}, u_{H_p-1}) - s_{H_p}\end{bmatrix} = 0.  
\end{equation}
Here, vector $x$ combines all the controls and shooting node variables $u_i$ and $s_i$ as  $x = (s_0^T, u_0^T, \dots, s^T_{H_p-1},
u^T_{H_p-1}, s_{H_p}^T)^T$, $\xi$ is the initial state $w_0(t)$ which is considered as a parameter, and $w(u_i, w_i)$ is the result of the
integration of the dynamics from $\tau_i$ to $\tau_i+\Delta\tau$ where we set $u(\tau) = u_i$ and $w(\tau_i) = s_i$.

The objective function \eqref{eq:obj_MPC} is approximated by 
\begin{eqnarray}\label{eq:obj_func}
f(x) && := \sum_{i=0}^{H_p-1}\left[(s_i - w_s)^TP(s_i - w_s) + (u_i - u_s)^TQ(u_i - u_s) \right] \nonumber\\
[-1.5ex]\\[-1.5ex]
&& + (s_{H_p} - w_s)^TS(s_{H_p} - w_s), \nonumber 
\end{eqnarray}
while the constraints are imposed only at $\tau =\tau_i$, the beginning of the intervals, as
\begin{equation}\label{eq:constr} 
s_i \in W, ~u_i \in U, ~ s_{H_p} \in R_T, (i=0, \dots, H_p-1).
\end{equation}
If we define $\Omega := U^{H_p}\times (W^{H_p} \times R_T)\subset \mathbb{R}^{n_x}$ then $\Omega$ is convex. Moreover, the objective
function \eqref{eq:obj_func} is convex quadratic. Therefore, the resulting optimization problem is indeed of the form \ref{eq:param_prob}.
Note that $\Omega$ is not a box but a curved convex set due to $R_T$.  

The nonlinear program to be solved at every sampling time has $4563$ decision variables and $4403$ equality constraints, which are expensive to evaluate due to
the ODE integration.

\subsection{Numerical simulation}
Before presenting the simulation results, we give some details on the implementation.
To evaluate the performance of the methods proposed in this paper we implemented the following algorithms:
\begin{itemize}
\item Full-NMPC -- the nonlinear program obtained by multiple shooting is solved at every sampling time to convergence by several SCP iterations.

\item PCSCP -- the implementation of Algorithm \ref{alg:A1} using the exact Jacobian matrix of $g$. 

\item APCSCP -- the implementation of Algorithm \ref{alg:A1} with approximated Jacobian of $g$. Matrix $A_k$ is fixed at $A_k = g'(x^0)$ for
all $k\geq 0$, where $x^0$ is approximately computed off-line by performing the SCP algorithm (Algorithm \ref{alg:A2}) to solve the
nonlinear programming \ref{eq:param_prob} with $\xi=\xi_0 = w_0(t)$.

\item RTGN -- the solution of the nonlinear program is approximated by solving a quadratic program obtained by linearizing the dynamics and
the terminal constraint $s_{H_p} \in R_T$. The exact Jacobian $g'(\cdot)$ of $g$ is used. This method can be referred to as a classical
real-time iteration \cite{Diehl2002b} based on the constrained Gauss-Newton method \cite{Bock1984,Deuflhard2004}.
\end{itemize}
To compute the set $R_T$ a mixed Matlab and C++ code has been used. The computed value of $r$ is $ 1.687836$, while the matrix $S$ is dense,
symmetric and positive definite.

The quadratic programs (QPs) and the quadratically constrained quadratic programming problems (QCQPs) arising in the algorithms
we implemented can be efficiently solved by means of interior point or other methods \cite{Boyd2004,Nesterov2004}. In our implementation, we
used the commercial solver \texttt{CPLEX} which can deal with both types of problems.

All the tests have been implemented in C++ running on a $16$ cores workstation with $2.7$GHz Intel\textregistered Xeron CPUs and $12$ GB of
RAM. We used \texttt{CasADi}, an open source C++ package \cite{Andersson2010} which implements automatic differentiation to calculate the
derivatives of the functions and offers an interface to \texttt{CVODES} from the \texttt{Sundials} package \cite{Serban2005} to integrate
the ordinary differential equations and compute the sensitivities. The integration has been parallelized using \texttt{openmp}. 

In the full-NMPC algorithm we perform at most $5$ SCP iterations for each time interval. We stopped the SCP algorithm when the relative
infinity-norm of the search direction as well as of the feasibility gap reached the tolerance $\varepsilon = 10^{-3}$.
To have a fair comparison of the different methods, the starting point $x^0$ of the PCA, APCA and RTGN algorithms has been set to the
solution of the first full-NMPC iteration.

The disturbance on the flows $q_{\mathrm{in}}$ and $q_{\mathrm{tributary}}$ are generated randomly and varying from $0$ to $30$ and $0$ to
$10$, respectively. 
All the simulations are perturbed at the same disturbance scenario.

We simulated the algorithms for $H_m = 30$ time intervals. The average time required by the four methods is summarized in Table
\ref{tb:timing}. 
\begin{table}[!ht]
\begin{center}
\caption{The average time of four methods}\label{tb:timing}
\vskip -0.5cm
\begin{tabular}{lrrrrr}\\ \hline
\texttt{Methods} \!\!\!&\!\!\!  \texttt{AvEvalTime[s]} \!\!\!&\!\!\! \texttt{AvSolTime[s]} \!\!\!&\!\!\! \texttt{AvAdjDirTime[s]}
\!\!\!&\!\!\!  \texttt{Total[s]} \\ \hline
Full-NMPC \!\!\!&\!\!\! 220.930~(91.41\%) \!\!\!&\!\!\! 20.748~(8.58\%)  \!\!\!&\!\!\! -               \!\!\!&\!\!\! 241.700 \\  
PCSCP     \!\!\!&\!\!\! 70.370~(90.05\%)  \!\!\!&\!\!\! 7.736~(9.90\%)   \!\!\!&\!\!\! -               \!\!\!&\!\!\!  78.142 \\  
RTGN      \!\!\!&\!\!\! 70.588~(96.97\%)  \!\!\!&\!\!\! 2.171~(2.98\%)   \!\!\!&\!\!\! -               \!\!\!&\!\!\!  72.795 \\
APCSCP    \!\!\!&\!\!\! 0.458~( 3.28\%)   \!\!\!&\!\!\! 11.367~(81.34\%) \!\!\!&\!\!\! 2.122~(15.18\%) \!\!\!&\!\!\!  13.975 \\
\hline  
\end{tabular}
\end{center}
\end{table}
Here, \texttt{AvEvalTime} is the average time in seconds needed to evaluate the function $g$ and its Jacobian;
\texttt{AvSolTime} is the average time for solving the QP or QCQP problems;
\texttt{AvAdjTime} is the average time for evaluating the adjoint direction $g'(x^k)^T{y}^k$ in Algorithm \ref{alg:A1}; 
\texttt{Total} corresponds to the sum of the previous terms and some preparation time. On average, the full-NMPC algorithm needed $3.27$ iterations to converge to a solution.

It can be seen from Table \ref{tb:timing} that evaluating the function and its Jacobian matrix costs $90\%-97\%$ of the total time. On the
other hand, solving a QCQP problem is almost $3-5$ times more expensive than solving a QP problem.
The computationally expensive step at every iteration is the integration of the dynamics and its linearization.
The computational time of PCSCP and RTGN is almost similar, while the time consumed in APCSCP is about $6$ times less than PCSCP. 

The closed-loop control profiles of the simulation are illustrated in Figures \ref{fig:u_1_4} and \ref{fig:u_5_10}. 
Here, the first figure shows the flows in the turbines and the ducts of lakes $L_1$ and $L_2$,
while the second one plots the flows to be controlled in the reaches $R_i$ ($i=1,\dots, 6$).
\begin{figure}[ht]
\centerline{\includegraphics[angle=0,width=10.0cm]{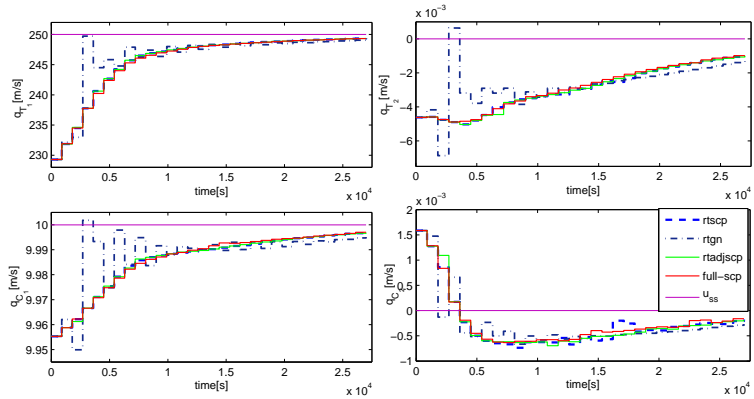}}
\caption{The controller profiles $q_{T_1}$, $q_{C_1}$, $q_{T_2}$ and $q_{C_1}$.}
\label{fig:u_1_4}
\end{figure}
\begin{figure}[ht]
\centerline{\includegraphics[angle=0,width=10.0cm]{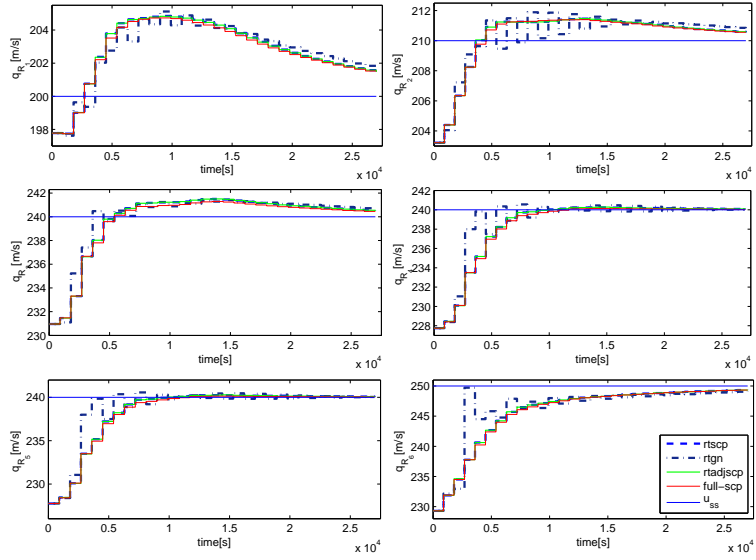}}
\caption{The controller profiles of $q_{R_1}, \dots, q_{R_6}$.}
\label{fig:u_5_10}
\end{figure}
We can observe that the control profiles achieved by PCSCP as well as APCSCP are close to the profiles obtained by Full-NMPC, while the
results from RTGN oscillate in the first intervals due to the violation of the terminal constraint. The terminal constraint in the PCSCP is active in many
iterations.

Figure \ref{fig:error} shows the relative tracking error of the solution of the nonlinear programming problem of the PCSCP, APCSCP and RTGN
algorithms when compared to the full-NMPC one.
\begin{figure}[ht]
\centerline{\includegraphics[angle=0,width=8.0cm]{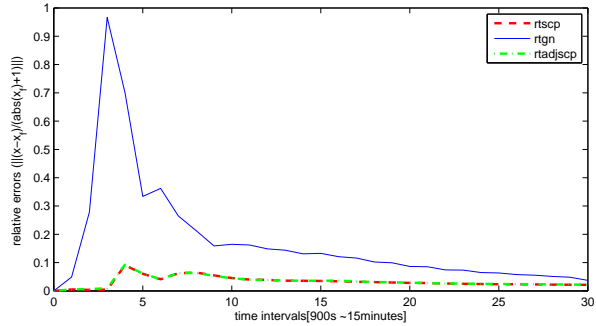}}
\caption{The relative errors of PCSCP, APCSCP and RTGN compared to Full-NMPC.}
\label{fig:error}
\end{figure}
The error is quite small in PCSCP and APCSCP while it is higher in the RTGN algorithm.
This happens because the linearization of the quadratic constraint can not adequately capture the shape of the terminal constraint $s_N\in
R_T$. The performance of APCSCP is nearly as good as PCSCP. This feature confirms the statement of Corollary \ref{co:tracking_corollary}.

\section{Conclusions}\label{sec:conclusion}
We have proposed an \textit{adjoint-based predictor-corrector SCP algorithm} and its variants for solving parametric optimization problems
as well as nonlinear optimization problems. We proved the stability of the tracking error for the online SCP algorithms and the local
convergence of the SCP algorithms. 
These methods are suitable for nonconvex problems that possess convex substructures which can be efficiently handled by using convex optimization
techniques \cite{Quoc2011a}. The performance of the algorithms is validated by a numerical implementation of an application in nonlinear model predictive
control. 
The basic assumptions used in our development are the strong regularity, Assumption \aref{as:A3}$\mathrm{b)}$ and Assumption
\aref{as:A3}$\mathrm{a)}$ (or \aref{as:A3'}). The \textit{strong regularity} concept introduced by Robinson in \cite{Robinson1980} and is  widely used in
optimization and nonlinear analysis, Assumption \aref{as:A3}$\mathrm{b)}$ (or \aref{as:A3'}) is needed in any Newton-type
algorithm. As in SQP methods, these assumptions involve some Lipschitz constants that are
difficult to determine in practice.

Our future work is to develop a complete theory for this approach and apply it to new problems. 
For example, in some robust control problem formulations as well as robust optimization formulations, where we consider worst-case
performance within robust counterparts, a nonlinear programming problem with second order cone and semidefinite constraints needs to be
solved that can profit from the SCP framework.

\section*{Acknowledgments}
Research supported by Research Council KUL: CoE EF/05 /006 Optimization in Engineering(OPTEC), 
IOF-SCORES4CHEM, GOA/10/009\\(MaNet), GOA/10/11, several PhD/postdoc and fellow grants; 
Flemish Government: FWO: PhD/postdoc grants, projects G.0452.04, G.0499.04, G.0211.05, G.0226.06, G.0321.06, G.0302.07, G.0320.08, G.0558.08, G.0557.08,
G.0588.09,G.0377.09, research communities
(ICCoS, ANMMM, MLDM); 
IWT: PhD Grants, Belgian Federal Science Policy Office: IUAP P6/04; 
EU: ERNSI; FP7-HDMPC, FP7-EMBOCON, Contract Research: AMINAL. 
Other: Helmholtz-viCERP, COMET-ACCM, ERC-HIGHWIND, ITN-SADCO.

\bibliographystyle{plain}

\end{document}

%% file: hpv_pgf_lib.tex
\newcommand{\lake}[3]{
\draw[fill=blue!20] #2 +(-1.75,0.5)-- +(-1,-1) -- +(1,-1) -- +(1.75,0.5);
\draw[line width=0.75pt] #2 +(-2,1)-- +(-1,-1) -- +(1,-1) -- +(2,1);
\node(#1) at #2 {#3};
\coordinate(#1_a) at ($#2+(0,0.5)$);
\coordinate(#1_b) at ($#2+(0,-1)$);
\coordinate(#1_br) at ($#2+(1,-1)$);
\coordinate(#1_bl) at ($#2+(-1,-1)$);
\coordinate(#1_l) at ($#2+(-1.5,0)$);
\coordinate(#1_r) at ($#2+(1.5,0)$);
}
\newcommand{\uduct}[3]{
\draw[line width=3pt,->](#1)--(#2);
\draw[fill=white, line width=0.75pt] ($0.5*(#1)+0.5*(#2)$) circle (0.6);
\node at ($0.5*(#1)+0.5*(#2)$) {#3};
}
\newcommand{\dduct}[3]{
\draw[line width=3pt,<->](#1)--(#2);
\draw[fill=white, line width=0.75pt] ($0.5*(#1)+0.5*(#2)$) circle (0.6);
\node at ($0.5*(#1)+0.5*(#2)$) {#3};
}

\newcommand{\dam}[3]{
\coordinate(#1_cc) at ($(#2)-(0.25,0.5)$);
\draw[fill=white, line width=0.75pt] (#1_cc) +(-1,0.75)-- +(-0.5,0.875) -- +(1.5,0.875) -- +(1.5,0) -- +(-1,0) -- cycle;
\draw[fill=white, line width=0.75pt] (#1_cc) +(-1,0.75)-- +(1,0.75) -- +(1,-0.75) -- +(-1,-0.75) -- cycle;
\draw[fill=white, line width=0.75pt] (#1_cc) +(1,0.75)-- +(1.5,0.875) -- +(1.5,-0.625) -- +(1,-0.75) -- cycle;
\node at ($(#1_cc)+(0,0.25)$) {#3};
\coordinate(#1_out) at ($(#1_cc)-(0,0.5)$);
}

\newcommand{\reach}[3]{
\draw[fill=blue!20, line width=0.75pt] ($(#2)+(0.5,-0.25)$) -- ($(#2)+(0.5,0.25)$) -- 
($(#2)+(-0.5,0.25)$) -- ($(#2)+(-0.5,0.25)+#3*(-1,-0.25)$) -- 
($(#2)+(-0.5,-0.25)+#3*(-1,-0.25)$) -- ($(#2)+(0.5,-0.25)+#3*(-1,-0.25)$) -- cycle;
\draw[line width=0.75pt] ($(#2)+(-0.5,0.25)+#3*(-1,-0.25)$) -- ($(#2)+(0.5,0.25)+#3*(-1,-0.25)$) -- ($(#2)+(0.5,0.25)$);
\draw[line width=0.75pt] ($(#2)+(0.5,0.25)+#3*(-1,-0.25)$) -- ($(#2)+(0.5,-0.25)+#3*(-1,-0.25)$);
\coordinate(#1_out) at ($(#2)+#3*(-1,-0.25)$);
\coordinate(#1_ma) at ($(#2)+0.5*#3*(-1,-0.25)+(-0.5,0.25)$);
\coordinate(#1_mb) at ($(#2)+0.5*#3*(-1,-0.25)+(0.5,-0.25)$);
}

\newcommand{\arrow}[2]{
\draw[line width=3pt,->](#1)--(#2);
}